\documentclass[12pt]{article}

\usepackage{amsmath}
\usepackage{amssymb}
\usepackage{amsfonts}

\usepackage{amsthm}

\newtheorem{prop}{Proposition}[section]
\newtheorem{theo}[prop]{Theorem}
\newtheorem{coro}[prop]{Corollary}
\newtheorem{lemm}[prop]{Lemma}

\def\N{\mathbb{N}}

\def\R{\mathbb{R}}
\def\E{\mathbb{E}}

\def\cX{\mathcal{X}}
\def\X{\mathcal{X}}

\def\cN{\mathcal{N}}
\def\cNstar{\mathcal{N}^*}

\def\cY{\mathcal{Y}}

\def\Pr{\mathbb{P}}

\def\0{{\bf 0}}

\def\cB{{\cal B}}

\renewcommand{\E}{\mathbb E \,}

\newcommand{\fmax}{f_{{\rm max}}}
\newcommand{\rhomax}{\rho_{{\rm max}}}
\newcommand{\rhomin}{\rho_{{\rm min}}}
\newcommand{\Cut}{{\rm Cut}}
\newcommand{\Div}{{\rm div}}
\newcommand{\Bal}{{\rm Bal}}
\newcommand{\CHE}{{\rm CHE}}
\newcommand{\MBIS}{{\rm MBIS}}

\newcommand{\eqco}{\setcounter{equation}{0}}
\newcommand{\allco}{\eqco}

\newcommand{\Vol}{{\rm Vol}}

\newcommand{\Y}{{\cal Y}}

\newcommand{\K}{{\cal K}}
\newcommand{\U}{{\cal U}}

\newcommand{\tphi}{{\tilde{\phi}}}
\newcommand{\tr}{{\tilde{r}}}

\newcommand{\tgamma}{{\tilde{\gamma}}}
\newcommand{\bvol}{{\overline{\Vol}}}
\newcommand{\uphin}{\phi_n}
\newcommand{\ophin}{\phi^{(n)}}

\newcommand{\eps}{\varepsilon}
\def\bdm{\begin{displaymath}}
\newcommand{\edm}{\end{displaymath}}
\def\benu{\begin{enumerate}}
\def\eenu{\end{enumerate}}
\def\beqn{\begin{equation}}
\def\eeqn{\end{equation}}
\def\be{\begin{equation}}
\def\ee{\end{equation}}
\def\bea{\begin{eqnarray}}
\def\eea{\end{eqnarray}}
\newcommand{\bean}{\begin{eqnarray*}}
\newcommand{\eean}{\end{eqnarray*}}
\newcommand{\bear}{\begin{eqnarray}}
\newcommand{\eear}{\end{eqnarray}}

\renewcommand{\epsilon}{\varepsilon}

\def\R{\mathbb{R}}

\def\qed{\hfill\hbox{${\vcenter{\vbox{
    \hrule height 0.4pt\hbox{\vrule width 0.4pt height 6pt
    \kern5pt\vrule width 0.4pt}\hrule height 0.4pt}}}$}}

\def\mod{{\| }}

\begin{document}
\title{\bf Optimal Cheeger cuts and bisections of random geometric graphs}

\author{
Tobias M\"uller$^{1}$
and
Mathew D. Penrose$^{2}$ \\
{\normalsize{\em Groningen University and University of Bath}} }

 \footnotetext{ $~^1$ Johann Bernoulli Institute of Mathematics and Computer Science, Groningen University, P.O. box 407, 9700 AK Groningen, The Netherlands:
{\texttt tobias.muller@rug.nl} }
 \footnotetext{ $~^2$ Department of
Mathematical Sciences, University of Bath, Bath BA2 7AY, United
Kingdom: {\texttt m.d.penrose@bath.ac.uk} }
 \footnotetext{ $~^1$ 
Research  partially supported
by NWO grants 639.032.529 and 612.001.409}
 \footnotetext{ $~^2$ 
Research partially supported by a STAR visitor grant, by 
NWO visitor grant 040.11.532, and by the Department
of Mathematics at Utrecht University. }
\footnotetext{Key words and phrases: random geometric graph, Cheeger constant, conductance.}
\footnotetext{AMS Classifications:  05C80, 60D05, 62H30}




\maketitle

\begin{abstract}
Let $d \geq 2$. 
The Cheeger constant of a graph is the minimum
surface-to-volume ratio of all subsets of the
vertex set with relative volume at most 1/2. 
There are several ways to define
surface and volume here:
the simplest method is
to count boundary edges (for the surface) and vertices (for the volume).
We show that for a geometric (possibly weighted) graph on
$n$ random points in a $d$-dimensional domain
 with Lipschitz boundary and with
distance parameter decaying more slowly
(as a function of $n$) than
the connectivity threshold, the Cheeger constant
(under several possible definitions of surface and volume), also known
as conductance, suitably rescaled, 
converges for large $n$ to an analogous Cheeger-type constant
of the domain. Previously,
Garc\'ia Trillos {\em et al.} had shown this for
 $d \geq 3$
but had required an extra condition on the distance
parameter when $d=2$.  
\end{abstract}

\section{Introduction}

A significant recent theme in topological/geometrical data analysis and machine learning
is the reconstruction of topological/geometrical properties of
a continuous space such as a manifold from a random sample of
points in that space via a graph, or more
generally a simplicial complex, derived
from the sample by connecting nearby points; see
for example
\cite{Carlsson,Cuevas,Edelsbrunner,estperim,Kahle}. 
A prototypical graph of this type is the {\em random geometric graph},
where one connects every pair of points up to a  specified
distance $r$ apart (we shall consider generalization
of this to allow for weighted graphs).

One quantity of considerable interest in both the continuum and discrete 
settings is the {\em Cheeger constant}. For a
$d$-dimensional Euclidean 
domain $D$ (or more generally, a manifold),
 the Cheeger constant is the minimum perimeter-to-volume
ratio of all subregions of $D$ with relative volume at most $1/2$
(here, when measuring the perimeter of a subregion  of
$D$, only the part of the boundary that is interior to $D$
is included).
It can be used to provide useful bounds for the first eigenvalue gap of the
Laplacian on $D$ (with Dirichlet boundary condition) \cite{Cheeger,Buser}.
The analogous quantity for a graph (there are several possible
definitions, as we shall describe below) similarly provides
  bounds for the eigengap of
the graph Laplacian, and is therefore important in, among other things,
  the study of the mixing time of a random walk on the graph 
  (see \cite{Benjamini,Kiwi} and \cite[Chapter 2]{Chung}, for example). Cheeger constants
provide a natural measure of the quality of the partition in cluster analyisis,
and are important in graph-based spectral clustering methods \cite{vonL}.

Given the above, it is of interest mathematically, but also
from the point of view of cluster analysis and machine learning,
 to know whether one can `learn'
about the Cheeger constant of the region $D$ from that of
the random geometric graph on a sample of points in $D$. 
More formally, is the discrete Cheeger constant 
based on a geometric graph on a sample of $n$ 
random points with distance parameter $r_n$, suitably rescaled,
 a consistent estimator
of the continuum Cheeger constant? If so,  for which
choices of the sequence $(r_n)_{n \geq 1}$
 is this the case? In practical terms,
 one would like to use
small values of $r_n$ to reduce the computational cost
of computing the Cheeger constant of the graph, but
if $r_n$ is too small then the graph will not be connected
and its Cheeger constant will be zero.
At least for regions $D$ with well-behaved boundary, the 
asymptotic threshold
for connectivity is known to be at $r_n = {\rm const.} \times 
((\log n)/n)^{1/d}$ \cite{RGG,LongMST}.

Such questions were first raised and partially answered by
Arias-Castro {\em et. al.} in 
\cite{APP}.
A more complete answer was given by Garc\'ia Trillos {\em et al.}
in \cite{consistency}; they established
consistency in dimensions $d \geq 3$ for all
sequences $(r_n) $ tending to zero more slowly than
the connectivity threshold,  but
left a  gap when $d=2$, as described
 in the next section.
In the present paper we provide an alternative proof
of consistency
  which fills this gap.  We also provide more detail than \cite{consistency}
for the proof in the case of some of the choices of definition
of Cheeger constant of a graph.

Our argument has the potential to 
provide a method of showing convergence
for a number of other graph optimization problems,
such as those described in \cite{VertOrd},
in the spirit of the celebrated BHH result 
 \cite{BHH} for the travelling salesman
problem. To illustrate
this, we also give a BHH type result for the the {\em minimal  bisection}
of a random geometric graph; that is, the partition of the vertices
into two {\em equal} pieces which minimises the total weight of cut edges. 
 Finding the minimal bisection is a classic problem in computer science with
applications, for example, in parallel processing and
Very Large Scale Integration; see for example
\cite{Bhatt,Diekmann}.

We briefly discuss some  of the sources of difficulty in these problems
and the techniques used to overcome them. 
In the bisection problem, for example, 
 the main difficulty is to find a 
good lower bound on
the cost of {\em all} possible
bisections of the point process.
By matching the points of the sample to those of a rectangular grid with
the same number of  points in $D$,
one may 
 identify
each such bisection with a bisection of the domain,
and  hence identify its cut weight with a suitably smoothed 
measure of the perimeter of the bisection of the domain.
Using a `liminf' Gamma-convergence bound from
\cite{cloud} one may
then asymptotically  lower bound the
cost of the point process bisection by 
the minimal perimeter of bisections of the domain.

Loosely speaking, this is the approach of \cite{consistency}. 
Its reliance on grid matching results means that
one requires $r_n$ to be larger than the 
distances involved in the grid matching, and in $d=2$ this
is known 
 to be a stronger condition than connectivity. 

Our contribution is to circumvent the need for 
any grid matching. To do this,
we use a coarser granulation of space into boxes
which are large enough for the number of points in a box to be
concentrated about its mean (but which are smaller than $r_n$).
We develop  a local optimization technique to show that
in every optimal bisection $(Y,Y^c)$ of the point set, 
each box contains mostly points of $Y$ or
 mostly points of $Y^c$  so every  optimal bisection of the point 
process may be identified with a collection of boxes whose union
bisects the domain. One may then use the Gamma-convergence
techniques as before.

\section{Statement of results}
\label{StateResult}

Let $d \in \{2,3,\ldots\} $ and let
$D \subset \R^d $ be open.
Let $(r_n)_{n \geq 1}$
be an $\R_+$-valued sequence, where $\R_+:=[0,\infty)$.
Let $X_1,X_2,\ldots$
be a sequence of independent random $d$-vectors taking values
in $D$ with common
probability density function denoted $\rho$.
 For $n \in \N$ let $\cX_n := \{X_1,\ldots,X_n\}$.
Given a function $\phi: \R_+ \to \R_+$,
let $G_\phi(\cX_n,r_n)$, the {\em $\phi$-weighted random geometric graph}
on the point set $\cX_n$ with distance
parameter $r_n$, be the complete graph on vertex set
$\cX_n$, with the weight of the  edge $\{X,Y\}$
given by $\phi(\mod Y-X\mod /r_n)$ for each $X,Y \in \cX_n$ with
$X \neq Y$,
where $\mod \cdot\mod $ denotes the Euclidean norm.

Two important special cases of $\phi$ are
\bean
\phi_U (t) := {\bf 1}_{[0,1]}(t); ~~~~~~
\phi_N (t) := \exp(-  t^2),
\eean
The graph
 $G_{\phi_U}(\cX_n)$
 amounts to the classic Euclidean (unweighted)
random geometric graph (also known as the Gilbert graph);
 see \cite{RGG} for an overview of such graphs.
The Gaussian (Normal) weight function $\phi_N$ is often used in
spectral clustering algorithms; see for example \cite{vonL}.
We shall consider a general class of $\phi$ satisfying
mild
 monotonicity and integrability conditions,
which includes the two examples just mentioned.

Given $\Y \subset X_n$, set
$$
\Cut_{n,\phi}(\Y):
 = \sum_{y \in \cY} \sum_{x \in \cX_n \setminus \cY} \phi
\left( \frac{ \mod x-y\mod  }{r_n} \right),
$$
which is the total weight  in
$G_\phi(\cX_n,r_n)$ of the {\em cut edges} induced
by $\Y$, i.e. the edges  from $\cY$ to its complement.
We are interested in choosing $\cY$ to make
$\Cut_{n,\phi}(\cY)$ small but with some penalty for choices
of $\cY$ for which $\cY$ or its complement is disproportionately
small.
 This penalty takes the form of dividing
$\Cut_{n,\phi}(\cY)$ by
a `balance term' based on the
`volume' of $\cY$ and its complement, where
`volume' may be measured by counting either vertices or
(weighted) edges. Several choices of balance term have thus been
proposed in the literature, including
\bea
\Bal_{n,v,1}(\cY) :=
\frac{ \min ( \Vol_{n,v}(\cY),\Vol_{n,v}(\cX_n \setminus \cY) )
}{
\Vol_{n,v}(\X_n)}, ~~~ v\in\{1,2\},
\label{1217b}
\eea
and 
\bea
\Bal_{n,v,2}(\cY) :=  \frac{
   \Vol_{n,v}(\cY)\Vol_{n,v}(\cX_n \setminus \cY)
}{ (\Vol_{n,v}(\X_n))^{2} }, ~~~ v\in\{1,2\}
,
\label{1217c}
\eea
where, with $|\cY|$ denoting the number of elements of $\cY$,
 we set
$$
\Vol_{n,1}(\cY) := |\cY|; ~~~~
 \Vol_{n,2}(\Y) := \sum_{y \in \Y}
 \sum_{x \in \X_n \setminus \{y\}} \phi \left( \frac{ 
\mod x-y\mod }{r_n }
\right).
$$
(Some authors include an extra factor of
 2 in the right hand side of (\ref{1217c}).)
In this paper we consider the Cheeger-type functionals
\bea
\CHE_{v,b}( G_\phi(\cX_n,r_n) ): = \min_{\cY \subset \cX_n: \cY
 \neq \emptyset, \cY \neq \cX_n}
\left( \frac{\Cut_{n,\phi}(\cY) }{\Bal_{n,v,b}(\cY)} \right),~~~ (v,b)
 \in \{1,2\}^2.
\label{Chee0}
\eea
The quantity being minimized in (\ref{Chee0}) is
sometimes called the {\em Cheeger cut} of $\Y$ for
$(v,b)=(1,1)$,  the {\em ratio cut} for $(v,b)=(1,2)$,
the {\em normalized cut} for $(v,b)=(2,1)$ and the
{\em sparsest cut} for $(v,b)=(2,2)$;
 see \cite{consistency} and references therein.
The terms {\em Cheeger constant} and {\em conductance} with reference
to a graph are used with little unanimity in the literature;
 all three of
 $\CHE_{1,1}(G_{\phi}(\cX_n,r_n))$,
 $\CHE_{2,1}(G_{\phi}(\cX_n,r_n))/\Vol_{n,2}(\X_n)$
and
 $\CHE_{2,2}(G_{\phi}(\cX_n,r_n))/(\Vol_{n,2}(\X_n))^2$
could be called
the  Cheeger constant or conductance of the graph
 $G_{\phi}(\cX_n,r_n)$; see for example \cite{consistency},
\cite{APP}, \cite{Benjamini}, \cite{Kiwi}.

One may also consider continuum analogues.
Let $\cB(D)$ denote the Borel $\sigma$-field on $D$.
 Let  $\nu$ be the measure on
$(D,\cB(D))$ with Lebesgue density $\rho$ (i.e., the distribution of $X_1$).\
  For $u \in L^1(\nu)$ set
$$
TV(u):= \sup \left\{ \int_D u(x) \Div (\Phi)(x) dx:
\Phi \in C_c^1(D: \R^d), |\Phi(x)| \leq \rho^2(x) \forall x \in D \right\}.
$$
Here  $C_c^1(D:\R^d)$ denotes the class of all continuously 
differentiable functions from $D$ to $\R^d$ having support
that is compact and contained in $D$. For $\Phi = (\Phi_1,\ldots,\Phi_d)
 \in C_c^1(D:\R^d)$, and $x = (x_1,\ldots,x_d) \in D$, we define
 $\Div (\Phi)(x) = \sum_{i=1}^d \frac{\partial \Phi_i}{\partial x_i} |_x$.  
We shall assume throughout that
 $D$ is bounded and connected, and that $D$ has {\em Lipschitz boundary},
which means that each $x \in \partial D$ (the boundary of $D$)
has a neighbourhood $U$ such that the restriction of $ \partial D $ to $ U$ is the graph
of a Lipschitz function after a suitable  rotation.
We shall also assume
 that  the density $\rho: D \to \R_+ $  is  continuous
 with $\rhomax := \sup_{x \in D} \rho(x) < \infty$
and $\rhomin := \inf_{x \in D} \rho(x) > 0$.
Then, according to (3.3) of \cite{cloud}
(see also Proposition 2.33 of \cite{Henrot}),
if $u={\bf 1}_A$
is the indicator function of a set $A \subset \R^d$ with $C^1$ boundary
(defined analogously to the Lipschitz boundary above), then
\bea
TV({\bf 1}_A) = \int_{\partial A \cap D} \rho^2(x)  {\cal H}^{d-1}(dx),
\label{0102a}
\eea
where ${\cal H}^{d-1} $ is the $(d-1)$-dimensional Hausdorff measure.
More generally than (\ref{0102a}), we assert 
for all $A \in \cB(D)$ with $TV({\bf 1}_A)< \infty$
that
$$
TV({\bf 1}_A) = \int_{\partial^* A \cap D} \rho^2(x) {\cal H}^{d-1}
(dx),
$$
where $\partial^*A$ is the De Giorgi  
reduced boundary of $A$ (see \cite[Definition 1.54]{Braides}).
This assertion follows from 
(3.5) of \cite{consistency} and Theorem 1.55 of \cite{Braides}.
We define the continuum Cheeger functionals  of $(D,\rho)$  by
\bea
\CHE_{v,b}(D,\rho) :=
 \inf_{A \in \cB( D): 0 < \nu(A) < 1}
\left( \frac{TV({\bf 1}_A) }{\Bal_{\nu,v,b}(A)  } \right),
 ~~~ (v,b) \in \{1,2\}^2,
\label{Chee1}
\eea
where for $A \in \cB(D)$ we set
$$
\Bal_{\nu,v,1} (A) := \frac{\min (
 \Vol_{\nu,v}(A),1- \Vol_{\nu,v}(D \setminus A) )}{\Vol_{\nu,v}(D)} ;
$$
$$
\Bal_{\nu,v,2} (A) :=
\frac{ \Vol_{\nu,v}(A) \Vol_{\nu,v}(D \setminus A) }{(\Vol_{\nu,v}(D))^2},
$$
 with
$$
\Vol_{\nu,v} (A) := \int_A (\rho(x))^v dx, ~~~ v \in \{1,2\}.
$$
Using (\ref{0102a}), it is easy to see for $(v,b) \in \{1,2\}^2$ that
 $\CHE_{v,b}(D,\rho) < \infty$.
 It is less easy  to see that
$\CHE_{v,b}(D,\rho) >0$ in general (but it is stated in
the MathSciNet review of \cite{Cheeger} that this is well-known
in the case where $\rho$ is constant, from which we can
deduce the same for general $\rho$), but we shall not use this.

It may be the case that in some circumstances,
 the definition of
 $\CHE_{v,b}(D,\rho)$
is unaffected by restricting the minimum
to sets $A$ with smooth boundary, for which
we can use the definition (\ref{0102a}) of $TV({\bf 1}_A)$.
A result along these lines
(for constant $\rho$ and under a further smoothness condition on $\partial D$)
 appears in \cite{CCN}, but to give such a result
in the generality considered here would be beyond
the scope of the present paper.

%

We shall assume $\phi$ satisfies the following conditions:
\bea
\phi(r) \geq \phi(s) ~~~ \forall r,s \in \R_+ ~{\rm with } ~ r \leq s ;
\label{phi1} \\
\phi(0) >0 ~~{\rm and} ~ \phi ~
 {\rm ~ is ~ continuous ~ on ~ }[0,\delta] {\rm ~for ~some ~} \delta >0;
\label{phi2} \\
\sigma_{\phi} := \int_{\R^d} \phi(\mod x\mod ) | 
x_1
|
 dx
< \infty,
\label{phi3}
\eea
where here $x_1$ denotes the first co-ordinate of $x$.
The quantity $\sigma_\phi$ is sometimes called the
 `surface tension' of $\phi$. In particular
 $\sigma_{\phi_U}$ is twice the quantity denoted $\gamma_d$ in
eqn (4) of
 \cite{APP}. We have
\bea
\sigma_{\phi_U} =
 \frac{2 \pi^{(d-1)/2}}{(d+1) \Gamma((d+1)/2) };
~~~~~
\sigma_{\phi_N} =
\pi^{(d-1)/2}.
\label{0105a}
\eea
The first identity of (\ref{0105a})  is derived in
\cite{APP}, and the second is standard.
For any two $\R_+$-valued sequences $(a_n)_{n \geq 1}$ and
$(b_n)_{n \geq 1}$ we write $a_n \gg b_n$ or $b_n \ll a_n$
or $b_n = o(a_n)$
if $\lim_{n \to \infty}(b_n/a_n) =0$ (defining $0/0:= 1$).
We also sometimes write
 $b_n= O(a_n)$ if $\limsup_{n \to \infty}(b_n/a_n) < \infty$,
and write $b_n =  \Theta (a_n)$ if both $b_n = O(a_n)$ and
$a_n = O(b_n)$.
We use the term `almost surely' (or `a.s.') to denote
`with probabilty one' (rather than `with probability tending to one').
The following is our main result.
\begin{theo}
\label{mainthm}
Suppose that $D$ is a nonempty, open, bounded, connected subset of $\R^d$
 with   Lipschitz boundary, and
 $\rho: D \to \R_+ $
is a continuous probability density function satisfying
 $\rhomax  < \infty$
and $\rhomin  > 0$.
Suppose that $\phi$ satisfies (\ref{phi1})-(\ref{phi3}), 
and that $(r_n)_{n \geq 1}$  satisfies $nr_n^d \gg \log n$
and $r_n \ll 1$. Let
 $(v,b) \in \{1,2\}^2$. Then
\bea
\label{1011a}
\lim_{n \to \infty} \left( \frac{ 
\CHE_{v,b}(G_\phi(\cX_n,r_n)) }{n^2 r_n^{d+1} }
\right)
 =
 (\sigma_{\phi}/2) \CHE_{v,b}(D,\rho),~~~ a.s.
\eea
\end{theo}
This was already shown by Garc\'ia Trillos {\em et al.} in 
\cite[Theorem 9]{consistency}, except that
 in the case $d=2$ they require the extra condition that
 $nr_n^2 \gg (\log n)^{3/2}$; our  result  answers a question
raised in Remark 2 of \cite{consistency} as to whether we
can do without  this extra
condition (yes we can). Moreover, in \cite{consistency}
 the proof is provided 
only for the case $v=1$.
Previously Arias-Castro et al. \cite{APP} asked about
the limiting behaviour when $(v,b)=(2,1)$.
To relate the case $(v,b)=(2,1)$ of the above
 result to the limiting behaviour of the
Cheeger constant as defined in \cite{APP}, note that
as a special case of Lemma \ref{vollem} below we have
\bea
\lim_{n \to \infty} (n^2 r_n^d)^{-1} \Vol_{n,2} (\X_n) = \int_D \rho(x)^2 dx
 \int_{\R^d} \phi( \mod y \mod ) dy.
\label{RGGedgeLLN}
\eea
 The case $\phi=\phi_U$ of (\ref{RGGedgeLLN})  was proved in Theorem 3.17
of \cite{RGG}. Note that the
 right hand side of (\ref{RGGedgeLLN}) is finite by
 (\ref{phi3}) and the assumptions on $D$ and $\rho$.

Our next theorem shows that under the
same hypotheses as in Theorem \ref{mainthm},
  the empirical measure of the optimising choice of $\cY$  in
(\ref{Chee0}) converges subsequentially
 to the restriction
of $\nu$ to an optimising set in the definition (\ref{Chee1}).
We use the standard notion of weak convergence
of probability measures on a metric space, as described in \cite{Bill},
for example.
Given $A \in \cB(D)$, let $\nu|_A$ denote the restriction of the measure $\nu $
to $A$, i.e. the measure on $D$ with density $\rho(\cdot) {\bf 1}_A (\cdot)$.
\begin{theo}
\label{Weaktheo}
Suppose the hypotheses of Theorem \ref{mainthm} hold.
Almost surely, for any sequence of minimisers $\cY_n$
in the definition (\ref{Chee0}) of $\CHE_{v,b}(G_\phi(\cX_n,r_n))$
and any infinite $\cN \subset \N$,
 there exists an infinite
 $\cN' \subset \cN$ 
and a  minimising set $A$ in the definition (\ref{Chee1}), 
such that we have
the weak convergence of measures
\bea
\sum_{y \in \cY_n} n^{-1} \delta_{y} \to \nu|_{A}
~~~~~ {\rm as} ~~n \to \infty ~~{\rm through}~~ \cN'.
\label{1018d}
\eea
\end{theo}
When the minimising set $A$ is essentially unique up to
complementation, one can re-phrase the preceding result without needing to
take subsequences,
as
follows.

\begin{coro}
\label{corowk}
Suppose that the hypotheses of Theorem \ref{mainthm} hold, and
also that the minimising set
$A$   in the definition (\ref{Chee1}) of $\CHE_{v,b}(D,\rho)$
is unique, up to complementation and adding or removing sets of
$(d-1)$-dimensional measure zero.

Then, almost surely, for any sequence of minimisers $\cY_n$
in the definition (\ref{Chee0}) of $\CHE_{v,b}(G_\phi(\cX_n,r_n))$
there exists a sequence  $(j(n), n \in \N) $ taking values in $\{0,1\}$,
such that setting $\overline{\cY}_n= \cY_n $ if $ j(n)= 1$ and
 $\overline{\cY}_n= \X_n \setminus \cY_n $ if $ j(n)=0$, 
we have
\bea
\sum_{y \in \overline{\cY}_n} n^{-1} \delta_{y} \to \nu|_{A}
~~~~~ {\rm as} ~~n \to \infty.
\label{0501a}
\eea
\end{coro}
For completeness, we shall provide a proof of
Corollary \ref{corowk} at the end of Section \ref{seclo}.

The {\em Prohorov distance} on probability measures on $D$ is
a metrization of weak convergence (see \cite{Bill}).
 Another interpretation of Theorem
\ref{Weaktheo} is that, almost surely, for any sequence of minimisers
${\cal Y}_n$ the Prohorov distance from 
 $\sum_{y \in \Y_n} n^{-1} \delta_y$
to  the set of measures of the form $\nu_A$ with $A$ a minimising set
in the definition (\ref{Chee1}) of $\CHE_{v,b}(D,\rho)$, tends to zero.

A result resembling Theorem \ref{Weaktheo} is provided in
\cite[Theorem 9]{consistency}, but again under the extra condition
 $nr_n^2 \gg (\log n)^{3/2}$ when $d=2$, and again with
proofs given only for $v=1$. Also, we use a different
(and apparently simpler) notion
of weak convergence of measures than the one used there.
 Both of these distinctions
are related to the fact that the proof in \cite{consistency}
proceeds via certain transportations of measures (we discuss this
further below).


Next we describe a similar result for the minimum bisection functional
\bea
\MBIS(G_\phi(\X_n,r_n)) := 
\min \{ \Cut_{n,\phi}(\cY): \Y \subset \X_n,|\cY| = \lfloor n/2 \rfloor \}
.
\label{0107a}
\eea
This functional has been considered
in \cite{Diaz}, \cite{RGG}  and elsewhere.
For  the regime considered here with $nr_n^d \gg \log n$ (in fact for a greater
range of regimes for $(r_n)$),
it was shown in \cite{VertOrd}
that for for $\phi=\phi_U$
under the additional
assumption that $D$ is the unit cube and $\nu$ is the uniform distribution
on $D$
that $\MBIS(G_{\phi_U}(\X_n,r_n))= \Theta(n^2 r_n^{d+1})$,
almost surely.
Under the further assumption that $d=2$ and using the $\ell_\infty$ distance to define
the random geometric graph, explicit upper and lower bounds are given
in \cite{Diaz} for
the limits superior and inferior of
 $\MBIS(G_{\phi_U}(\X_n,r_n))/ (n^2 r_n^{d+1})$ which  differ by a factor of 4.
In this section we give a
BHH-type result
for this problem (for general $d$ and $D$, using the Euclidean distance),
i.e. a strong law for
$ \MBIS(G_\phi(\cX_n,r_n)) $ in the regime
 $nr_n^d \gg \log n$. The result goes as follows.
\begin{theo}
\label{thmBISBHH}
Suppose the hypotheses of Theorem \ref{mainthm} hold.
Then
\bea
\label{1018e}
\lim_{n \to \infty} \left(
 \frac{ \MBIS(G_\phi(\cX_n,r_n)) }{n^2 r_n^{d+1} } \right)
= (\sigma_{\phi}/2)
\MBIS_\nu(D),
~~~ a.s.
\eea
where we set
\bea
\MBIS_\nu(D) :=
\inf_{A \in \cB(D): \nu(A) = 1/2} TV({\bf 1}_A).  
\label{MBISDdef}
\eea
\end{theo}
One might also consider these problems for geometric graphs on
 other sequences of point process
besides the binomial point process $\cX_n$. For example, the
results should carry through if instead
of $\X_n$ one considered a Poisson point process $\X_{N_n}$ with
$N_n$ Poisson($n$) distributed and independent of
$(X_1,X_2,\ldots)$. They should carry through because the 
main probabilistic tools used are the  Chernoff bounds
(\ref{CherUB}) and (\ref{CherLB}) for the
binomial distribution, and analogous bounds
 are also available for the Poisson distribution.

Another possibility 
 would be to consider instead of $\X_n$ a
deterministic rectilinear grid with spacings of size $n^{-1/d}$.
In this case one might be able to get the same results
with the condition $n r_n^d \gg \log n$ weakened to
$n r_n^d \gg 1$.

Another natural extension of the results would
be to consider Riemannian  manifolds, which is
the setting of the original work of Cheeger \cite{Cheeger}.
We have not attempted this but it seems likely that our
methods can be extended to the manifold setting.
One reason to include non-uniform $\rho$ in our results is 
that this may be useful in extending them the manifold setting.

For the rest of this paper we assume that $D,\rho$ and $\phi$ satisfy
the conditions assumed in the statement of Theorem \ref{mainthm}.
We also assume that $r_n \ll 1$, and that  $nr_n^d \gg \log n$.

Here is an overview of our method of proof of the
`liminf' part of (\ref{1011a}) and (\ref{1018d}).
We divide $D$ into boxes of side $\gamma_n r_n$, where $\gamma_n \to 0$ slowly.
 By Chernoff bounds (Lemma \ref{lemCher}),
the number of points in each box is close to 
its expected value. Given an optimal subset $\cY \subset \X_n$,
 we adjust  $\Y$ to a set $\Y' \subset \X_n$
that is not too different from $\cY$, such that
 all boxes have mostly points in $\Y'$ or
mostly points in $\X_n \setminus \Y'$, and which is also close
to optimal. Then we approximate to
$\Y$ by the union of boxes containing mostly points
of $\Y'$ and estimate the discrete cut of $\cY$ by
an approximation to the perimeter
for  this union of boxes.
We then use a Gamma-convergence result from \cite{cloud}
(Lemma \ref{LemGam} below) to derive the desired liminf inequality.

The method of \cite{consistency} is related, but
relies (via the paper \cite{cloud})
 on results of Shor {\em et al.} \cite{Leighton,ShorY},
extended in \cite{empirical}
to the class of domains considered here, on the existence 
of a matching of a grid of side $\Theta(n^{-1/d})$
 to the random point set $\X_n$,
with maximum displacement at most
 $O(((\log n)/n)^{1/d})$ (for $d \geq 3$) or 
$O(((\log n)^{3/2}/n)^{1/d})$ (for $d=2$). 
(It also requires a
notion of weak convergence of pairs $(\mu_n,T_n)$
where $\mu_n$ is a measure and $T_n$ a functional.)
As mentioned earlier, our method avoids relying
on grid matchings enabling us,
when $d=2$, to relax the condition
$r_n \gg (\log n)^{3/4}n^{-1/2}$ 
(required in \cite{consistency}) to
$r_n \gg (\log n)^{1/2}n^{-1/2})$.

%
%
%
%

We have included the cases with $v=2$ in our proofs.
This entails extra work; see for example Lemmas
\ref{lemEgoroff}  and \ref{lemBnice}.
We suspect that a similar amount of work would
be needed to fill in the details of proof for
$v=2$ of the corresponding results in  \cite{consistency}.

\section{Preliminaries}
\label{secprelim}
\allco

Since we assume that $n r_n^d \gg \log n$ and $r_n \ll 1$, we can and do choose a sequence $(\gamma_n)_{n \geq 1}$ of
constants, 
 such that $\gamma_n \ll 1$ and also
\bea
n  \gamma_n^{d+2} r_n^d \gg \log n;     ~~~~~~~~~~~~~~
\gamma_n^{d+4} \gg r_n.
\label{gammacond}
\eea
In other words the $\gamma_n$ tend to zero, but possibly very slowly.

 Given $n \in \N$, divide $\R^d$ into half-open rectilinear cubes
$Q'_{1,n},Q'_{2,n},\ldots$ of side $\gamma_n r_n$,
with the centre of $Q'_{i,n} $ denoted $z_{i,n}$. To be definite,
assume the origin is one of the points $z_{i,n}$.
Let $S_n := \{i \in \N: Q'_{i,n} \subset D\}$ (which is a nonempty set for large enough $n$),
and let $D_n:= \{z_{i,n}:i \in S_n\}$.

Suppose $n$ is such that $S_n  \neq \emptyset$.
For $j \in \N$,
let $i=I(j,n) \in S_n$ be chosen so that $z_{i,n}$ is the nearest point of
 $D_n$ to $z_{j,n}$, using the lexicographic ordering on
$\R^d$ to break any ties. (In particular, if $j \in S_n $ then $I(j,n) =j$.)
Then for each $i \in S_n$,  define the set
\bea
Q_{i,n}  := \cup_{\{j: i = I(j,n) \}}  (Q'_{j,n} \cap D).
\label{1018a}
\eea
That is, $Q_{i,n} $ is the union of $ Q'_{i,n}$ itself,
and those boundary cubes $Q'_{j,n}$ which have $Q'_{i,n}$ as the nearest
interior cube (intersected with $D$).
We shall refer to the sets $Q_{i,n}, i \in S_n$ as {\em boxes},
even though only those sets $Q_{i,n}$ lying
 away from the boundary of $D$ are necessarily cubes.

Since $D$ has Lipschitz boundary, using a compactness argument
 we can find  constants
$C \geq d$ and  $n_0 \in \N$
such that for all $n \geq n_0$,
 for all boundary boxes
$Q_{j,n}$ there is an interior box within distance $(C/3) \gamma_n$,
and hence
\bea
\label{1121a}
\mod x-y\mod  \leq C \gamma_n r_n, ~~~
 \forall ~i \in S_n,  x,y \in  Q_{i,n} .
\eea
For each $n \in \N$ we define a function
 $\uphin(x,y)$ (respectively $\ophin(x,y)$) that approximate  the
 weight function
$\phi (\mod x-y \mod /r_n)$ from below (respectively, from above)
 and is constant on each product of boxes, as follows:
for each $i,j \in S_n$
set
\bea
\phi_n(x,y):=
 \inf_{x' \in Q_{i,n},y' \in Q_{j,n}} \phi( \mod x' -y' \mod /r_n) ,~~~
x \in Q_{i,n}, y \in Q_{j,n};
\label{phindef}
\\
\ophin(x,y):=
 \sup_{x' \in Q_{i,n},y' \in Q_{j,n}} \phi( \mod x' -y' \mod /r_n) ,~~~
x \in Q_{i,n}, y \in Q_{j,n}.
\label{phinupdef}
\eea
%
\begin{lemm}
\label{lemphin}
There exist
 constants $C' \in (0,\infty)$, $n_1 \geq n_0$,
and $(\tilde{\gamma}_n)_{n \in \N}$ with $\tilde{\gamma}_n
\ll 1$,
 depending only on $D$ and $\phi$,
such that for all $n \geq n_1$ and all $x,y \in D$ we have
\bea
\phi( \mod x-y \mod /r_n) \geq \phi_n(x,y) \geq
(1- \tgamma_n) \phi( \mod x-y \mod /\tr_n)
\label{1121d}
 \eea
and
\bea
\phi( \mod x-y \mod /r_n) \leq \ophin(x,y) \leq
(1 + \tgamma_n) \phi( \mod x-y \mod /r'_n),
\label{1128a}
\eea
where we set $\tr_n := (1- C' \gamma_n) r_n$ and
 $r'_n := (1+ C' \gamma_n) r_n$.
\end{lemm}
\begin{proof}
The first inequality of (\ref{1121d}) is clear from the definition
(\ref{phindef}).
To prove the second inequality,
observe that for any $i,j \in S_n$, for $x,x' \in Q_{i,n}$ and
$y,y' \in Q_{j,n}$, by (\ref{1121a})
we have $\mod  y'-x'\mod  \leq \mod y-x\mod + 2C \gamma_n r_n$, so that
using (\ref{phi1}) we have
\bea
\phi_n(x,y) \geq
\phi
\left( \frac{\mod x-y\mod }{r_n}
+ 2 C \gamma_n \right).
\label{1122a}
\eea
Using assumption  (\ref{phi2}), choose
 $a >0$ with
 $\phi $
continuous (and hence uniformly
continuous) on the interval $[0,2a]$ and
$\phi(2a) >0$.
 Then
by the uniform continuity,
there is a function $h: \R_+ \to \R_+$ with
$h(u) \to 0$ as $u \downarrow 0$ such that for
all $t,u \in [0,a]$  we have
 $\phi(t+u) \geq \phi(t) - h(u)$.

For $0 \leq t \leq a$, and for $n$ large enough so that
$ \gamma_n \leq a/(2C)$, we have that
$$
\frac{\phi(t+2 C \gamma_n)}{\phi(t)} \geq
\frac{ \phi(t) - h(2C \gamma_n)}{\phi(t)}
\geq 1 - \frac{h(2 C \gamma_n)}{\phi(2a)}.
$$
Setting $\tgamma_n : = h(2C \gamma_n)/\phi(2a)$, by
(\ref{1122a}) we have for
$\mod x-y\mod  \leq a r_n$ that
\bea
\phi_n(x,y)
 \geq (1- \tgamma_n) \phi \left( \frac{\mod x-y\mod }{r_n} \right)
 \geq (1- \tgamma_n) \phi \left( \frac{\mod x-y\mod }{\tr_n} \right).
\label{1121e}
\eea
Take $C' > 2C/a$. Then
for $t > a r_n$, if $n$ is large enough so that $C' \gamma_n <1$ we have
$$
\frac{t}{ (1- C' \gamma_n)r_n} > (t/r_n)(1+ C' \gamma_n ) > (t/r_n) + a C' \gamma_n
> (t/r_n) + 2 C \gamma_n.
$$
Hence by (\ref{1122a}), for $\mod x-y\mod  > a r_n$ we have
$$
\phi_n(x,y) \geq \phi (
\mod x -y \mod /
 \tr_n
 )
\geq (1- \tgamma_n) \phi(
\mod x-y \mod
/ \tr_n
).
$$
Combined with (\ref{1121e}) for $\mod x-y\mod  \leq a r_n$, this
gives us (\ref{1121d}).

The proof of (\ref{1128a}) is similar.
\end{proof}

For $n \in \N$ and $p \in [0,1]$  let
 ${\rm Bi}(n,p)$ denote a binomial random variable
with parameters $n ,p$. Also let $H(x) = 1-x+ x \log x$ for $x >0$, and let $H(0)=1$.
The following  Chernoff-type  bounds are well-known
(see for example Lemma 1.1 of \cite{RGG}):
\bea
\Pr[ {\rm Bi}(n,p)  \geq k] \leq  \exp \left( -n p H \left( \frac{k}{np} \right)  \right), ~~~~~ k \geq np;
\label{CherUB}
\\
\Pr[ {\rm Bi}(n,p)  \leq k] \leq  \exp \left( -n p H \left( \frac{k}{np} \right)  \right), ~~~~~ k \leq np.
\label{CherLB}
\eea
\begin{lemm}
\label{lemCher}
There exists an almost surely finite random variable $N$
such that for all $n \geq N$ and all $i \in S_n$ we have
\bea
|\cX_n \cap Q_{i,n}| & \leq & (1+\gamma_n) n \nu(Q_{i,n})
\label{Cher1}
\eea
and
\bea
|\cX_n \cap Q_{i,n}| \geq (1- \gamma_n)  n \nu(Q_{i,n}).
\label{Cher2}
\eea
\end{lemm}
\begin{proof}
By Taylor's theorem, for $x \in \R$ with $ \mod x \mod $ sufficiently small
we have that $H(1+x)> (1/3)x^2$, and hence
for large enough $n$ we have
for all $i \in S_n$
by (\ref{CherUB}) that
\bea
\Pr[ |\cX_n \cap Q_{i,n}| >  (1+\gamma_n) n \nu(Q_{i,n}) ] \leq \exp(- n \nu(Q_{i,n})  H(1+\gamma_n))
\nonumber \\
\leq \exp( - n \rhomin r_n^d \gamma_n^{d+2}/3 )
\label{1026a}
\eea
and by (\ref{gammacond}) this bound is $O(n^{-9})$. Since $D$ is bounded
and $n (\gamma_n r_n)^d \to \infty$ by (\ref{gammacond}),
we have that $|S_n| = O(n)$. Therefore it
follows by (\ref{1026a}), the union bound and the Borel-Cantelli lemma that (\ref{Cher1}) holds for
all but finitely many $n$, almost surely. The proof of (\ref{Cher2}) is similar,
this time using (\ref{CherLB}).
\end{proof}


We shall repeatedly use the following result due to
Garc\'ia Trillos and Slep\v{c}ev
\cite{cloud}.
Given $r >0$ define the functional $TV_{\phi,r}$ on $L^1(\nu)$ by
\bea
TV_{\phi,r}(u)
 := r^{-d-1} \int_D \int_D \phi \left(\frac{ \mod x-y \mod }{r}\right) |u(x) - u(y)| \nu(dx) \nu(dy),
\label{TVrdef}
\eea
 as in (1.9) on page 203 of \cite{cloud}
(see page 195 of \cite{cloud} for the definition of $\phi_r$ used there).
Note that $TV_{\phi,r} (au) = aTV_{\phi,r}(u)$ for all $a >0$.
For $u$ of the form
 $u = {\bf 1}_A$ for some $A \in \cB(D)$, 
the functional $TV_{\phi,r}(u)$ may be 
viewed as providing a smoothed measure of the perimeter of $A$.
\begin{lemm}
\label{LemGam}
 Let $(\eps_n)_{n \geq 1}$ be a
 $(0,1)$-valued sequence with $\eps_n \ll 1$. 
Then: \\
 (i) {\rm [liminf lower bound]} for any $L^1(\nu)$-valued sequence $(u_n)_{n \geq 1}$
converging in $L^1(\nu)$ to some $u  \in L^1(\nu)$,  
we have
\bea
\liminf_{n \to \infty} TV_{\phi,\eps_n} (u_n) \geq \sigma_\phi TV(u).
\label{LimInfEq}
\eea
(ii) For any $u \in L^1(\nu)$,
\bea
\lim_{n \to \infty} TV_{\phi,\eps_n} (u) = \sigma_\phi TV(u).
\label{RecovEq}
\eea
 (iii) {\rm [Compactness]}
If
$(u_n)_{n \geq 1}$
is an $L^1(\nu)$-valued sequence
that is bounded in $L^1(\nu)$, and
 $TV_{\phi,\eps_n}(u_n)$ is bounded,
then there is a subsequence $(n_k)$ along which $u_{n_k} \to u  $
in $L^1(\nu)$ for some $u \in L^1(\nu)$.
\end{lemm}
\begin{proof}
Part (i)  is from
 Theorem 4.1 of \cite{cloud} and the definition of Gamma-convergence
(also in \cite{cloud}, for example).
Part (ii) is from Remark 4.3 of \cite{cloud}.
Part (iii) is also from Theorem 4.1 of \cite{cloud}.
\end{proof}

For the reader's convenience,
we offer the following clarifications to \cite{cloud},
kindly provided by Nicol\'as Garc\'ia Trillos.
In (4.9) of that paper, the constant $C$ needs to be
allowed to depend on $\delta$ but this does not affect the
subsequent argument there. Also, the
 first display of page 230 of \cite{cloud} is incorrect;
one can avoid needing to use this display
 by changing $\eta_{\eps}$ to $\eta_{\eps/4}$ in
(4.22) of \cite{cloud} and then changing $\eta_\eps$ to $\eta_{\eps/4}$
and $\eta_{4 \eps}$ to $\eta_\eps$ throughout page 229
of \cite{cloud}.


\section{Upper bound}
\label{secupper}
\allco

Throughout this section we assume $D$, $\rho$, $\phi$ and $(r_n)_{n \geq 1}$
satisfy the assumptions of  Theorem $\ref{mainthm}$.
We prove the following result,
which is the easier half of Theorem \ref{mainthm}.
\begin{prop}
\label{propupper}
Given $(v,b) \in \{1,2\}^2$, we have
\bea
\limsup_{n \to \infty} \left( \frac{ \CHE_{v,b}(G_\phi(\cX_n,r_n))}{
n^2 r_n^{d+1}} \right)
\leq (\sigma_\phi/2) \CHE_{v,b}(D,\rho).
\label{1013b}
\eea
\end{prop}
The proof of this in the case $v=2$ requires the following result
which also justifies our earlier assertion
 (\ref{RGGedgeLLN}):
\begin{lemm}
\label{vollem}
Let $A \in \cB(D)$. Then, almost surely,
\bea
\lim_{n \to \infty} (n^2 r_n^d)^{-1} \Vol_{n,2} (\X_n \cap A) =
 \int_A \rho(x)^2 dx
 \int_{\R^d} \phi( \mod y \mod ) dy.
\label{RestrEdge}
\eea
Moreover, there exists a  constant $C'' \in (0,\infty)$
such that a.s., for all large enough $n$ and all $\cY \subset \cX_n$,
\bea
(C'')^{-1} n r_n^d | \cY| \leq
\Vol_{n,2}(\cY) \leq 
C'' n r_n^d | \cY|.
\label{1216a}
\eea
\end{lemm}
\begin{proof}
Let $N$ be as in Lemma \ref{lemCher}.
Then for all $n \geq N$, using
 the first inequality of (\ref{1128a}), then 
 (\ref{Cher1})
followed by the second inequality of
 (\ref{1128a}),
we have
for all $X \in \X_n$ that
\bea
\Vol_{n,2}(\{X\}) & \leq & (1+ \gamma_n) n \int \ophin(X,y) \nu(dy)
\nonumber \\
& \leq & (1+ \gamma_n) (1 + \tgamma_n) n \int \phi \left( \frac{
\mod y-X \mod }{ r'_n }\right) \nu(dy).
\label{1128b}
\eea
Also, using (\ref{1121d}) and (\ref{Cher2}) 
we have for all $n$ large enough so that $n \geq N$
  and $\gamma_n n \nu(Q_{i,n}) \geq  1$ for all $i \in S_n$,
 and all $X \in \X_n$, that
\bea
\Vol_{n,2}(\{X\}) & \geq & (1- 2 \gamma_n) n \int \phi_n(X,y) \nu(dy)
\nonumber \\
& \geq & (1-2 \gamma_n) (1- \tgamma_n) n \int \phi \left( \frac{
 \mod y-X\mod }{ \tr_n} \right) \nu(dy).
\label{1128c}
\eea
Defining $\rho(\cdot) \equiv 0$ on $\R^d \setminus D$ and recalling
the definition of $r'_n$ from Lemma \ref{lemphin},
for $x \in D$ let us set
\bea
h_n(x) := r_n^{-d} \int_{\R^d} \phi \left(\frac{\mod y-x \mod }{ 
r'_n } \right) \rho(y) dy
= \left( \frac{ r'_n}{ r_n }\right)^d
 \int_{\R^d} \phi(\mod u\mod ) \rho(x + r'_n u) du.
\label{hneq}
\eea
 By (\ref{1128b}), it
is almost surely the case that for large enough $n$  we
have
\bea
\Vol_{n,2}(\X_n \cap A)
\leq
(1+ \tgamma_n) (1+ \gamma_n)n  \sum_{X \in \X_n \cap A}
 r_n^d h_n(X).
\label{1128e}
\eea
Since $\rho$ is continuous on $D$,
 $\rhomax < \infty$ and
 $I_\phi:= \int_{\R^d} \phi( \mod x \mod )dx $ is finite by (\ref{phi3}),
by dominated convergence we have $h_n(x) \to \rho(x) I_\phi =: h(x) $
for all $x \in D$.
Moreover $h_n(x)$ is bounded
 uniformly in $(n,x)$.
 Therefore by a version of
the strong law of large numbers, 
$$
\lim_{n \to \infty}
n^{-1} \sum_{i=1}^n h_n(X_i) {\bf 1}_A(X_i) =
 \E[h(X_1) {\bf 1}_A(X_1)] =
 I_\phi \int_A \rho(x)^2 dx.
$$
Hence by (\ref{1128e}) we have
\bea
\limsup_{n \to \infty} n^{-2} r_n^{-d}
\Vol_{n,2}(\X_n \cap A) \leq
 I_\phi \int_A \rho(x)^2 dx.
\label{1128d}
\eea
A similar argument using
(\ref{1128c}) instead of (\ref{1128b})
shows that
$$
\liminf_{n \to \infty} n^{-2} r_n^{-d}
\Vol_{n,2}(\X_n \cap A) \geq
 I_\phi \int_A \rho(x)^2 dx,
$$
and together with (\ref{1128d}) this
gives us (\ref{RestrEdge}) as asserted.

Finally, observe that the proof of 
(\ref{1128e}) above shows also
for all $\cY \subset \X_n$
 that 
$$
\Vol_{n,2}(\cY) \leq (1+ \tgamma_n) (1+ \gamma_n) n \sum_{X \in \cY}
r_n^d h_n(X).
$$
Since $h_n$ is bounded uniformly in $x $ and $n$,
this implies the second inequality of (\ref{1216a}).
The first inequality of (\ref{1216a}) is obtained similarly
from  (\ref{1128c}): note that $h_n$ can be shown to be uniformly
bounded away from zero, using the assumptions that
 $\rhomin >0$ and $D$ has Lipschitz boundary. 
%
%
\end{proof}
By adapting the proof of Lemma \ref{vollem},
we can obtain the following, which will be used in Section
\ref{seclo}. The sequence $(\gamma_n)_{n \geq 1}$ is chosen as described
at the start of Section  \ref{secprelim}. We shall need the factor
of $(1-2 \gamma_n)$ (rather than just $1-\gamma_n$) in (\ref{0108d})
when we use this result later on.

\begin{lemm}
\label{lemEgoroff}
Almost surely, the following holds. For
any infinite $\cN \subset \N$ and any
 sequence $(I_n)_{n \in \cN}$ of subsets of $S_n$
such that the set $B_n:= \cup_{i \in I_n} Q_{i,n} $ satisfies
$\inf_{n \in \cN } \nu(B_n) >0$, and
any sequence $(\U_n)_{n \in \cN} $ of subsets of $\X_n \cap B_n$
with
\bea
(1- 2 \gamma_n) n \nu(Q_{i,n}) \leq |\U_n \cap Q_{i,n}| \leq 
(1 + \gamma_n) n \nu(Q_{i,n}),
 ~~~ \forall i \in I_n, n \in \cN,
\label{0108d}
\eea
it is the case that
\bea
\lim_{n \to \infty,n \in \cN} \left( \frac{\Vol_{n,v}(\U_n)
\Vol_{\nu,v}(D) }{
\Vol_{n,v}(\cX_n)
\Vol_{\nu,v}(B_n) 
 } \right) =1, ~~~ v = 1,2.
\label{0108b}
\eea
\end{lemm}
\begin{proof}
The result is trivial for $v=1$, so
it suffices to consider the case with $v=2$.
Given $\eps >0$, let $D_{\eps}$ denote the
set of points $x \in D$ lying at Euclidean
distance at least $\eps$ from $\R^d \setminus D$.
Then $D_{\eps/2}$ is compact so $\rho $
is uniformly  continuous on $D_{\eps/2}$. Set
$$
B_{n,\eps} : = \cup \{Q_{i,n}: i \in I_n, Q_{i,n} \cap D_{\eps} \neq \emptyset
 \}, 
$$
which is contained in $D_{\eps/2}$ for large enough $n$.

In the proof of Lemma \ref{vollem},  note that
if $X \in \X_n$ and $Q_{i,n}$ is the box containing $X$,
 then for any $z \in Q_{i,n}$  the inequalities
 (\ref{1128b}) hold with the $X$ on
the right replaced by $z$ (both times) because
the function $\ophin$ is constant on products of boxes.
Therefore setting 
$$
\underline{h}_n(x) := \inf_{z \in Q_{i,n}} h_n(z) ,~~~ x \in Q_{i,n},
$$
we have as in (\ref{1128e}) that 
\bean
\Vol_{n,2}(\U_n \cap B_{n,\eps}  ) & \leq &
(1+ \tgamma_n) (1+ \gamma_n)n r_n^d  \sum_{X \in \U_n \cap B_{n,\eps} }
  \underline{h}_n(X)
\\
& \leq & 
(1+ \tgamma_n) (1+ \gamma_n)^2n^2 r_n^d
\int_{ B_{n,\eps} } 
  \underline{h}_n(x) \nu(dx),
\eean
where the last line comes from
 (\ref{0108d}).
 
As discussed in the  proof of Lemma \ref{vollem}, we have $h_n(x) \to h(x)
:= \rho(x) I(\phi)$
for all $x \in D$, where $I(\phi):= \int \phi( \mod z \mod )dz$. Moreover, we have
the uniform convergence
 $\sup_{x \in B_{n,\eps}} |h(x) - \underline{h}_n(x)| \to 0$.
This can be seen from (\ref{hneq}), using the fact that $\rho$ is uniformly
continuous on $D_{\eps/2}$ (most easily by first
considering the case where $\phi$ has bounded support).
We
therefore have that
\bea
\Vol_{n,2}(\U_n \cap B_{n,\eps}   ) \leq
(1+o(1)) n^2 r_n^d  \int_{B_{n,\eps} } h(x) \nu(dx) .
\label{0108a}
\eea
Since $\nu(D \setminus D_\eps) \to 0$ as $\eps \downarrow 0$,
given  $\delta >0$,
using (\ref{1216a}) 
 we may choose $\eps >0$ so that
$$
\limsup_{n \to \infty} (n^{-2}r_n^{-d} 
\Vol_{n,2} (\U_n \setminus B_{n,\eps}))
 < \delta.
$$ 
Combined with (\ref{0108a}) and using the assumption
that $\nu(B_n)$ is bounded away from zero for $n \in \N$,
 this shows that
\bea
\limsup_{n \to \infty, n \in \N} \left( \frac{\Vol_{n,2}(\U_n)}{n^2 r_n^d \int_{B_n} 
h(x)\nu (dx) } \right) \leq 1.
\label{0108c}
\eea
By a similar argument one may show an inequality  the other
way for the limit inferior, and therefore the fraction in the 
left hand side of (\ref{0108c}) actually tends to 1
as $n \to \infty$ through $\cN$. Hence by
(\ref{RGGedgeLLN}),
\bean
 \frac{\Vol_{n,2}(\U_n)
\Vol_{\nu,2}(D) }{
\Vol_{n,2}(\cX_n)
\Vol_{\nu,2}(B_n) 
 } \sim \frac{\Vol_{n,2}(\U_n)}{n^2 r_n^d I(\phi) 
\int_{B_n}\rho(x)^2 dx
 } 
\to 1
\eean
as $n \to \infty$ through $\cN$.
Thus we have
the case $v=2$ of (\ref{0108b}).
\end{proof}

\begin{lemm}
\label{thm1}
Let $A \in \cB(D)$ with $0< \nu(A)<1$.
Let $(v,b) \in \{1,2\}^2$.
For $n \in \N$  set  $\cY_n := \cX_n \cap A$.
Then as $n \to \infty$,
\bean
(n^2 r_n^{d+1})^{-1}
 \left( \frac{ \Cut_{n,\phi}(\cY_n)}{ \Bal_{n,v,b}(\cY_n)  } \right)
\to \frac{(\sigma_\phi/2) TV({\bf 1}_A)}{ \Bal_{\nu,v,b}(A) } , ~~~a.s.
\eean
\end{lemm}
\begin{proof}
By the strong law of large numbers (for $v=1$) or
by Lemma \ref{vollem} (for $v=2$),
\bean
\lim_{n \to \infty}  \Bal_{n,v,b}(\cY_n) = \Bal_{\nu,v,b}(A),
~~~ (v,b) \in \{1,2\}^2.
\eean
%
Therefore it suffices to show that
\bea
(n^2 r_n^{d+1})^{-1}
  \Cut_{n,\phi}(\cY_n)
\to (\sigma_\phi/2) TV({\bf 1}_A)
, ~~~a.s.
\label{1103a}
\eea
The convergence of expectations corresponding to (\ref{1103a}) follows
from taking $u = {\bf 1}_A$ in Part (ii) of   Lemma \ref{LemGam}.

The almost sure convergence in (\ref{1103a})
was proved in \cite{estperim} for
the case where $\rho$ is constant on $D$ and $\phi=\phi_U$. In
 Remark 1.10 of \cite{estperim}
 it is stated that the proof carries through to more general $\rho$
and to all weight functions $\phi$ satisfying (\ref{phi1})-(\ref{phi3}).
 A similar result, with a more restricted range of sequences $(r_n)$ than
we consider here,
is given in Theorem 1 of \cite{APP}.

The result (\ref{1103a})
 can alternatively  be proved using a similar argument to the
 proof of Theorem 3.17 of \cite{RGG}, at least when $\phi$ has bounded
support.
\end{proof}

\begin{proof}[Proof of Proposition \ref{propupper}]
Immediate from Lemma \ref{thm1}.
\end{proof}

\section{Lower bound}
\label{seclo}
\allco

In this section we complete the proof of
 Theorem \ref{mainthm}.
We shall also prove Theorem \ref{Weaktheo}.
Let $D$, $\rho$, $\phi$ and $(r_n)_{n \geq 1}$
be given, satisfying
 the assumptions of 
 Theorem \ref{mainthm}.
Let $(v,b) \in \{1,2\}^2$.
 If $\CHE_{v,b}(D,\rho)=0$ then (\ref{1011a}) is
immediate from Proposition \ref{propupper}, so we assume until
the end of the proof of Theorem \ref{mainthm} that
$\CHE_{v,b}(D,\rho) > 0$.


Our argument is related to one seen in in \cite{Diaz}.
Given $\cY \subset \X_n$,
think of points in $\cY$ as being {\em black} and points of 
$\cX_n \setminus \cY$ as being {\em white}.
 For $n \in \N$ and $i \in S_n$ (defined in Section \ref{secprelim}),
 we shall say that the box
 $Q_{i,n}$ is  {\em grey} (with respect to $\Y$) if  both
the number of black  points in $Q_{i,n}$, and the number of white points in $Q_{i,n}$,
exceed $ \gamma_n n \nu(Q_{i,n})$.
 We shall say the box $Q_{i,n} $ is {\em black}
(with respect to $\cY$) if it is not grey and
$ |\cY \cap Q_{i,n} | \geq (1- 2 \gamma_n) n \nu(Q_{i,n})$.
We shall say $Q_{i,n} $ is {\em white}
(with respect to $\cY$) if it is not grey and
$|\cX_n \setminus \cY| \cap Q_{i,n} \geq (1- 2 \gamma_n) n \nu(Q_{i,n}) $.
By (\ref{Cher2}), for $n \geq N$ every box is either black, white or grey.

Let $g_n(\cY)$ denote the number of grey boxes with respect to $\cY$.
 In other words, set
  $$
g_n(\cY): = \sum_{i\in S_n}
{\bf 1} \{  \min ( |\cY \cap Q_{i,n}|, |(\cX_n \setminus \cY) \cap Q_{i,n}|) >
\gamma_n n \nu(Q_{i,n}) \}.
$$

Define the {\em within-box}  edges of $G_\phi(\cX_n,r_n)$ to be
 those edges $\{x,y\}$
 such that
$\{x,y\} \subset Q_{i,n}$
 for some $i$
(i.e., such that both endpoints lie in the same box), and let all other edges
of $G_\phi(\cX_n,r_n)$ be called {\em between-box} edges.

By (\ref{1121a}) there exists $n_2 \in \N$ such that for $n \geq n_2$,
every within-box edge has weight at least $\phi(0)/2$ in
$G_\phi(\X_n,r_n)$, that is
\bea
\phi \left( \frac{ \mod x-y \mod }{r_n}\right)
\geq \frac{\phi(0)}{2}, ~~ \forall i\in S_n,x,y \in Q_{i,n}, n \geq n_2
\label{1216b}
\eea

 Let  $(\cY_n)_{n \in \N} $ be a sequence of non-empty proper 
subsets of $\cX_n$,
each of which satisfies
$\Vol_{n,v}( \cY_n) \leq \Vol_{n,v}(\X_n)/2$
and achieves the minimum in
 (\ref{Chee0}), i.e.
\bea
\frac{  \Cut_{n,\phi}(\cY_n)}{\Bal_{n,v,b}(\Y_n)}
=
 \CHE_{v,b}(G_\phi(\cX_n,r_n)).
\label{1121f}
\eea
\begin{lemm}
\label{LemYn}
Almost surely, it is the case that
\bea
\limsup_{n \to \infty}
\left(
\frac{ \Cut_{n,\phi}(\cY_n)}{ n^2r_n^{d+1} \Bal_{n,v,b}(\cY_n)} \right) \leq
(\sigma_\phi/2)
\CHE_{v,b}(D,\rho),
\label{1013c}
\eea
and that there exists $n_3 \in \N$
such that for all $n \in \N$ with
$n \geq n_3 $,
  at least one box is black
and at least one box is white
with respect to $\Y_n$.
\end{lemm}
\begin{proof}
The first statement (\ref{1013c}) follows from
 (\ref{1121f}) and (\ref{1013b}).

First suppose $v=1$.
Suppose for infinitely many $n$ that
there  is no black box with respect to $\Y_n$.
Then every box is grey or
white, so each vertex in $\cY_n$
has at least $\gamma_n \rhomin n (\gamma_n r_n)^d $ within-box white
 neighbours, and therefore by (\ref{1216b}),
\bean
\frac{ (n^2 r_n^{d+1})^{-1} \Cut_{n,\phi}(\cY_n) }{ 
(| \cY_n|/n) } \geq   \gamma_n^{d+1}
 \rhomin r_n^{-1} \phi(0)/2,
\eean
and by (\ref{gammacond}) this
 contradicts (\ref{1013c}), whether we take $b=1$ or $b=2$,
since the right hand side of (\ref{1013c}) is finite. Similarly,
for large enough $n \in \N$ at least one box is white.

Now suppose instead that $v=2$. 
By (\ref{1216a}) and (\ref{RGGedgeLLN}),
$$
\frac{\Vol_{n,2} (\cY_n)}{\Vol_{n,2} (\cX_n)} = O \left(
\frac{|\cY_n|}{n} \right),
$$
so we can argue similarly to the case already considered.
\end{proof}

We now define a modification $\bvol_{n,v}$ of the set function $\Vol_{n,v}$
with better linearity properties.
For $x,y \in  D$, 
recalling the definition (\ref{phindef}) of the function $\phi_n(x,y)$,  define
$$
\tphi_n  (x,y) = 
\begin{cases}
0 & {\rm if ~}  x,y \mbox{ lie in the same box} 
\\
\phi_n(x,y) & \mbox{otherwise.}
\end{cases}
$$
For $\cY \subset \X_n$, define
\bea
\overline{\Vol}_{n,v}(\cY) := \begin{cases}
\sum_{y \in \cY} 
\sum_{x \in \X_n \setminus \{y\}} \tphi_n(x,y)
 & {\rm if ~} v=2
\\
\Vol_{n,1}(\cY) 
 & {\rm if ~} v=1.
\end{cases}
\label{bvoldef}
\eea
\begin{lemm}
\label{lemtv}
There exist constants $\delta_n \downarrow 0$ such that
almost surely, for all $n$ and all $\cY \subset \cX_n$,
\bea
\Vol_{n,v}(\cY) \leq (1+ \delta_n) \overline{\Vol}_{n,v} (\cY), ~~~~ v=1,2. 
\label{1217a}
\eea
\end{lemm}
\begin{proof}
It suffices to consider the case $v=2$.
For $x \in D$, $\alpha \in (0,\pi)$ and 
$e \in \R^d$ with $ \mod e \mod  =1$, let $\K(x,e,\alpha)$
denote the open cone consisting of those $y \in \R^d \setminus \{x\}$
 such that the vector $y-x$ makes an angle less  
than $\alpha $ with $e$. For $r>0$ let
$B(x;r) := \{y \in \R^d: \mod y-x \mod  \leq r\}$, and
let $\omega_d$ denote the Lebesgue measure of $B(x;1)$.

By the assumption that $D$ has Lipschitz boundary,
and a  compactness argument, we can (and do) choose $\alpha \in (0,1/6)$
and $r_0 >0$ such that for all $x \in D $ 
there exists $e(x) \in \R^d$ with $ \mod e(x) \mod =1$ such that
$\K(x,e(x),3 \alpha) \cap B(x;r_0) \subset D$.

Choose $a \in( 0,1/4)$ with $\phi(2a) >0$. For $x \in D$, and 
$n \in \N$ large enough so that $a r_n < r_0/2$, 
note that $B(x + ar_ne(x) ;3 a \alpha r_n) \subset D$.
Then for large enough $n$ and all $x \in D$,
for every $i \in S_n$ such that
 $Q_{i,n} \cap B(x+ a r_ne(x) ; a \alpha r_n) \neq \emptyset$, we have
 that $Q_{i,n} \subset D$ and moreover
$Q_{i,n} = Q'_{i,n}$ (that is, $Q_{i,n}$ does not
touch the boundary of $D$), and
 furthermore 
 for all $y \in Q_{i,n}$
we have
$(a/2)r_n \leq  \mod y -x  \mod  \leq  2 a r_n$
so that $\phi_n(x,y) \geq \phi(2a)$. 
Therefore summing over all such $i$ 
and using (\ref{Cher2}), we obtain
for all $x \in \X_n$
 that
$$
\bvol_{n,2}(\{x\}) \geq (1- \gamma_n) n 
\rhomin \phi(2a) \omega_d (a \alpha r_n)^d,
$$
while using (\ref{1121a}) and (\ref{Cher1}) we have that
$$
\Vol_{n,2}(\{x\})
- \bvol_{n,2}(\{x\})
 \leq 2 n \rhomax (2C \gamma_n r_n)^d. 
$$
Summing over $x \in \cY$ we obtain that
$$
\frac{
\Vol_{n,2}(\{\cY\}) - \bvol_{n,2}(\{\cY\})
}{
\bvol_{n,2}(\{\cY\})
}
\leq  \frac{2^{1+d}  \rhomax C^d \gamma_n^d  }{(1- \gamma_n) \rhomin \phi(2a)
 \omega_d (\alpha a)^d }
$$
which tends to zero, as required.
\end{proof}

For $n \in \N$ and $\Y \subset \X_n$, define the modified cut function
\bea
\Cut'_{n,\phi}(\Y) := \sum_{x \in \cY}
 \sum_{y \in \cX_n \setminus \cY} \phi_n(x,y).
\label{1121g}
\eea

Denote by
$\cNstar$ the set of  $n \in \N$ such that
 $\Vol_{n,v}(\Y_n)/\Vol_{n,v}(\X_n) \in [\gamma_n,1/2 ]$.
The next lemma is a key part of our proof. It provides
a method of `greyscale removal' whereby we 
modify $\cY$ slightly in a manner that makes
all the boxes black or white.
\begin{lemm}
\label{Apartheid}
Almost surely, there exists a sequence
 of subsets $\Y'_n \subset  \X_n$, defined for $n \geq 1$, 
satisfying
\bea
n^{-1} |\cY'_n \triangle \cY_n|
\ll \gamma_n^2,
\label{1013e}
\eea
and a (random) number $n_4 \in \N$ such that
 for all $n \in \N $ with $n \geq  n_4$ we have that
\bea
g_n(\cY'_n) =0,
\label{1127c}
\eea
and that the union $B_n$ of black boxes
induced by $\cY'_n$ satisfies $B_n \neq \emptyset$ and
$D \setminus B_n \neq \emptyset$,
and moreover  that
\bea
 \Cut_{n,\phi}(\Y_n)
\geq
\Cut'_{n,\phi}( \cY'_n)
~~~~ {\rm if~} n \in \cNstar
\label{1127a}
\eea
and
\bea
\frac{\Cut_{n,\phi}( \cY_n)}{\bvol_{n,v}(\cY_n)} \geq
 \frac{Z_n}{\bvol_{n,v}(\cY'_n \cap B_n)}~~~ {\rm if~}
 n \in \N \setminus \cNstar,
\label{1127b}
\eea
where $Z_n$ denotes the
 contribution to $ \Cut'_{n,\phi}(\cY'_n)$
from edges with exactly one endpoint in $B_n$.
\end{lemm}

\begin{proof}
Set $K_n := 5 \rhomin^{-2} \gamma_n^{-2d -2} \sigma_\phi 
\CHE_{v,b}(D,\rho)/\phi(0)$,
which tends to infinity since
$\gamma_n$ tends to zero.
Suppose there is an infinite set $\cN_1 \subset \N$ such that
$g_n(\cY_n) \geq K_n r_n^{1-d}$
for all $n \in \cN_1$.
 Then by considering only the
within-box edges and using (\ref{1216b}),
 we have for large enough $n \in \cN_1$ that
\bean
\Cut_{n,\phi}(\cY_n) \geq g_n(\cY_n) (\gamma_n^{d+1} \rhomin n r_n^d)^2
\phi(0)/5
\\
 \geq K_n \gamma_n^{2d +2} \rhomin^2 n^2 r_n^{d+1} \phi(0)/5 \\
= n^2 r_n^{d+1} \sigma_\phi \CHE_{v,b}(D,\rho),
\eean
which would contradict (\ref{1013c}) since $\Bal_{n,v,b}(\cY_n) \leq 1$.
 Therefore  there exists
$n_5 \in [ n_0,\infty)$ such that
\bea
g_n(\cY_n ) < K_nr_n^{1-d} {\rm ~  for ~ all ~}   n \geq n_5 .
\label{1129a}
\eea

Suppose $\cNstar$ is infinite and let
 $n \in \cNstar$ (so that
$\Vol_{n,v}(\cY_n)/\Vol_{n,v}(\cX_n) \in [\gamma_n,1/2]$).
Assume also  that $n \geq \max( N, n_5)$, where
 $N$ is as in Lemma \ref{lemCher}.
We consider the effect of changing the colour of some of the vertices
 in a given box, on the contribution
  of  between-box edges  to the cut. Let $i \in \N$. Suppose there
are $\ell $ black vertices (with respect to $\cY_n$)
 and $w$ white vertices in the box $Q_{i,n}$, and
recalling the definition of $z_{i,n}$ from the start of
Section \ref{secprelim},
set
\bea
\ell' := \sum_{j \in S_n \setminus \{i\} }
\phi_n(z_{i,n},z_{j,n})
|\Y_n \cap Q_{j,n}|
;
\label{1121b}
\\
w' := \sum_{j \in S_n \setminus \{i\} }
\phi_n(z_{i,n},z_{j,n})
|(\X_n \setminus \Y_n) \cap Q_{j,n}|.
\label{1121c}
\eea
Let $m:= \ell+w $, the total number of vertices in $Q_{i,n}$.
Then the total contribution 
to $\Cut'_{n,\phi}(\cY_n)$ from between-box edges with one endpoint in $Q_{i,n}$ is equal to
the expression
\bean
\ell w' + w \ell' = \ell w' + (m-\ell) \ell' = \ell(w'-\ell') + m\ell'.
\eean
This expression is a linear function of $\ell$, if we consider $m,w'$ and $\ell'$ as
being fixed.  Therefore as a function of $\ell$ it is minimised
over the range $[0,m]$ either at $\ell=0$ or at $\ell =m$ (or both).
Moreover, taking $\ell=0$ or $\ell=m$ reduces the number of
within-box edges in this box to zero.
Hence we can (and do) modify the colour of vertices in $Q_{i,n}$ to make
 all vertices in $Q_{i,n}$ have the same colour, in such a way that
resulting set $\tilde{\cY_n}$ of black vertices  has
$\Cut'_{n,\phi}(\tilde{\cY_n}) \leq \Cut'_{n,\phi}(\cY_n)$.

Repeating this process for each of the $i$ such that $Q_{i,n}$ is grey with
respect to
the original set $\cY_n$, considered one by one,
we end up with a new set of black vertices, denoted 
$\cY'_n$, with $g_n(\cY'_n)=0$,
such that $\Cut'_{n,\phi}(\cY'_n) \leq \Cut'_{n,\phi}(\cY_n)$.
Also  
 $\Cut'_{n,\phi}(\cY_n) \leq \Cut_{n,\phi}(\cY_n)$ by (\ref{1121g})
 and (\ref{1121d}),
so we have (\ref{1127a}).

Next we prove (\ref{1127b}).
Suppose $\N \setminus \cNstar$ is infinite (else (\ref{1127b}) holds vacuously for
large enough $n$).
Let $n \in \N \setminus \cNstar$.
Assume $n \geq n_3$ with
$n_3$ given by Lemma \ref{LemYn}, so there is at least  one black box
and at least one white box with respect to $\cY_n$.
Let us write
$x_{n,1}$ for the total weight (using weight function $\phi_n$)  of between-box cut edges  involving black vertices in black boxes,
$x_{n,2}$ for the total weight  of between-box cut edges  involving black vertices in grey boxes
and $x_{n,3}$ for the total weight  of within-box cut edges  involving black vertices in white boxes.
Let $y_{n,1},y_{n,2},y_{n,3}$ be the total volume $\bvol_{n,v}$ (as defined
in (\ref{bvoldef})  of the set 
 of black vertices in black boxes, in grey boxes
and in white boxes respectively. 
Set $V_n := \Vol_{n,v}(\X_n)$.

Then
\bean
\frac{\Cut'_{n,\phi} (\cY_n) }{\bvol_{n,v}(\Y_n)/V_n} \geq \frac{V_n x_{n,1} + V_n x_{n,2} + V_n x_{n,3}}{y_{n,1} + y_{n,2}
+ y_{n,3}}
\geq \min \left(
\frac{ V_n x_{n,1} + V_n x_{n,2}}{y_{n,1} + y_{n,2}}, \frac{V_n x_{n,3}}{y_{n,3}} \right).
\eean
By
 a similar argument to the proof of Lemma \ref{LemYn}
 (see also Lemma \ref{vollem}),
we have that $(n^2 r_n^{d+1})^{-1}  (V_n x_{n,3}/y_{n,3})$ exceeds
a strictly  positive
constant times $\gamma_n^{d+1} r_n^{-1}$, and therefore by (\ref{1013c}), for
large enough $n \in \N \setminus \cNstar$   the above minimum
must be achieved by the the first of the two ratios.

We now look again
at the effect of changing the the colour of  vertices in a grey box $Q_{i,n}$.
With $w'$ defined by
 (\ref{1121c}),
let $\ell''$ be defined similarly to $\ell'$ in
(\ref{1121b}) but with the sum restricted to those
$j \in S_n \setminus \{i\}$ for which the box
$Q_{j,n}$ is black or grey. Then
set $\alpha := w'- \ell''$.
Write $x$ for the expression denoted $x_{n,1}+x_{n,2}$ above
and $y$ for $y_{n,1} + y_{n,2}$.
If we change the number of black vertices in the box by amount $k$,
keeping the total number of vertices the same
(in fact we shall consider just two possible values of $k$ below), then
the value of $x$ changes to $x+ \alpha k:=x'$ and $y$ changes to $y + \beta k: = y' $,
where we set
$$
\beta = \begin{cases}
1 & \mbox{ if } v=1
\\
\sum_{j \in S_n} \tphi_n ( z_{i,n},z_{j,n}) |\Y_n \cap Q_{j,n}| & \mbox{ if }
v=2.
 \end{cases}
$$
Then
$$
\frac{x'}{y'} - \frac{x}{y} = \frac{x+ \alpha k}{y+ \beta k} - \frac{x}{y}
= \frac{(x+ \alpha k)y - x(y+ \beta k)}{(y+ \beta k)y}
= \frac{k(\alpha y -  \beta x)}{(y+ \beta k)y},
$$
which  can be made  non-positive either by taking
$k= |(\X_n \setminus \cY_n) \cap Q_{i,n}|$ or by taking 
$k= - |\cY_n \cap Q_{i,n}|$
 (depending on the sign of $\alpha y - \beta x$).
Note that since there is at least one black box and every 
box has at least one neighbouring box,
 we have $y + \beta k > 0$ for
both of these choices of $k$.

Therefore we can choose a colour (white or black) and change all the vertices in
$Q_{i,n}$ to that colour, without increasing the ratio $x/y$.
Repeating this for each of the grey boxes in turn, we end up with a set
$\cY'_n$ that induces no grey boxes
and has
 a reduced (or at least not increased) value of $x/y$ compared to $\cY_n$.
Also, this procedure does not affect the value of $x_{n,3}$ or $y_{n,3}$
because, while some new white boxes might be created, none of them
contains black vertices at all.
Let $y'_{n,1}$ be the number of black vertices in black boxes induced by $\cY'_n$,
and let $x'_{n,1}$ be the total $\phi_n$-weight of between-box cut 
edges involving these vertices.  Then
$$
\frac{\Cut_{n,\phi} (\cY_n)}{\bvol_{n,v}(\cY_n)/V_n}
\geq \frac{V_n x_{n,1}+ V_n x_{n,2}}{ y_{n,1}+  y_{n,2}} \geq
\frac{V_n x'_{n,1}}{y'_{n,1}},
$$
and (\ref{1127b}) follows because $x'_{n,1} \geq Z_n$.

Finally, in both cases considered above ($n \in \cNstar $ and
$n \in \N \setminus \cNstar$), the modification of $\cY_n$ to obtain $\cY'_n$
involves changing the colour only of vertices in grey boxes
(with respect to $\cY_n$), so every black (repectively white)
box with respect to $\cY_n$ is also black (resp. white) with
respect to $\cY'_n$. It then follows from Lemma \ref{LemYn}
that almost surely, $B_n$ and $D \setminus B_n$
are non-empty for large enough $n$.
Moreover, for large enough $n$, by (\ref{1129a})
and (\ref{Cher1}),
\bean
n^{-1} |\cY'_n \triangle \cY_n|  & \leq &
 2 K_n
 \gamma_n^d  \rhomax  (2C)^d  r_n
\nonumber \\
& = &
O( \gamma_n^{-d-2}  r_n  ).
\eean
Thus by
(\ref{gammacond})
 we have (\ref{1013e}).
\end{proof}

As in the statement of Lemma \ref{Apartheid},
for $n \in \N$ let $B_n$
denote the union of the black boxes induced by
$\cY'_n$,
 let $Z_n $ be the contribution to $ \Cut'_{n,\phi}(\cY'_n)$
from edges with one endpoint in $B_n$ and the other endpoint in
$W_n := D \setminus B_n$, the union of the white boxes.
Then by (\ref{Cher2}),
$$
Z_n \geq (1- 2 \gamma_n)^2 n^2 \left(\int_{B_n} 
\int_{W_n} \phi_n(x,y) \nu(dx) \nu(dy)
\right). 
$$
Therefore by
Lemma \ref{lemphin}, with
 $C'$, $\tgamma_n$ and  $\tr_n$ as defined in that result, 
\bea
Z_n \geq (1- 2 \gamma_n)^2(1- \tgamma_n) n^2 \left(\int_{B_n} \int_{W_n}
 \phi \left( \frac{ \mod y-x \mod }{\tr_n} \right) \nu(dx) \nu(dy)
\right) \nonumber
\\
= (1- 2 \gamma_n)^2 (1- \tgamma_n) n^2 \tr_n^{d+1} ((1/2)
 TV_{\phi,\tr_n}({\bf 1}_{B_n})    ),
\label{1015c}
\eea
where $TV_{\phi,r}(u)$
 is as defined in (\ref{TVrdef}).
\begin{lemm}
\label{Npplem}
Almost surely
 $\N \setminus \cNstar$ is finite.
\end{lemm}
\begin{proof}
Suppose $\N \setminus \cNstar$ is infinite. Let $n \in \N \setminus \cNstar$, so
$\Vol_{n,v}(\cY_n)<  \gamma_n \Vol_{n,v}(\X_n) $.
Set $V_n := \Vol_{n,v}(\X_n)$.
By (\ref{1217b}), 
 (\ref{1217c}) 
and (\ref{1217a}), 
$$
\Bal_{n,v,b}(\cY_n) \leq \Vol_{n,v}(\Y_n)/V_n \leq (1+ \delta_n) \bvol_{n,v}
(\cY_n)/V_n.
$$
Therefore using
(\ref{1121f})
and
(\ref{1127b}),
 we have for $b=1,2$ that
\bea
\CHE_{v,b}(G_\phi(\X_n,r_n))
 \geq
 \frac{
(1+ \delta_n)^{-1} 
\Cut_{n,\phi} (\cY_n)}{
\bvol_{n,v} (\cY_n)/V_n}
\geq
\frac{ 
(1+ \delta_n)^{-1}
 Z_{n}}{ \bvol_{n,v}( \cY'_n \cap B_n)/ V_n  }.
\label{1017c}
\eea
By (\ref{Cher1}), if $v=1$ then 
$\Vol_{n,1}(\cY'_{n} \cap B_n) \leq 2 n 
\nu (B_n)$ for $n$ large;
if $v=2$ then using (\ref{1216a}) and
(\ref{RGGedgeLLN})
we can find a constant $C'''$ such that
for all large enough $n$ we have
$$
\Vol_{n,2}(\cY'_{n} \cap B_n)/V_n \leq
\frac{C'' |\Y'_n \cap B_n|}{n \int \rho(x)^2 dx \int \phi( \mod y \mod )dy}
\leq
 C''' (1+ \gamma_n) \nu (B_n).
$$
  Hence by (\ref{1015c}) and (\ref{1017c}),
for every $(v,b) \in \{1,2\}^2$
there exists a constant $c>0$ such that for all large enough $n$,
\bean
\frac{ \CHE_{v,b}(G_\phi(\X_n,r_n))}{
n^2 r_n^{d+1}}
\geq \frac{ 
c    TV_{\phi,\tr_n}({\bf 1}_{B_n}) }{2 \nu(B_n)}.
\eean
Since the functions $ {\bf 1}_{B_n}/\nu(B_n)$ are  
 $L^1(\nu)$-bounded, uniformly in $n$,
this shows  
by (\ref{1013b}) and compactness (part (iii) of Lemma \ref{LemGam}) that
 the functions ${\bf 1}_{B_n}/\nu(B_n)$ converge in $L^1(\nu)$ along
a subsequence to
a limiting function of the form ${\bf 1}_B/\nu(B)$ with
 $B \in \cB(D)$ and $0 < \nu(B)< 1$; see Lemma 6 of \cite{consistency}.

However, by (\ref{1013e})
and (in the case $v=2$) (\ref{1216a}) 
we have that
$n^{-1} |\cY'_n| \to 0$ as $n \to \infty$ through $\N \setminus \cNstar$.
Then by (\ref{Cher2}), as $n \to \infty $ through $\N \setminus \cNstar$ we have
$$
\nu(B_n) \leq (1- \gamma_n)^{-1} n^{-1} |\Y'_n| \to 0,
$$
which contradicts the conclusion above that
${\bf 1}_{B_n}/\nu(B_n)$ converges in $L^1$ to a limit
of the form ${\bf 1}_B$, for some  Borel $B \subset D$ with
  $0 < \nu(B) < 1$. Therefore $\N \setminus \cNstar$ must be finite.
\end{proof}

\begin{lemm}
\label{lem0106}
Let $\cY'_n$ and $B_n$ be as in Lemma \ref{Apartheid}. 
Set $\cY'_{n,B}:= \cY'_n \cap B_n$. Then
\bea
 \frac{ TV_{\phi, \tr_n} ({\bf 1}_{B_n})}{
\Bal_{n,v,b}(\cY'_{n,B}) } 
\leq
( 2 +o(1) )
 \left(\frac{\CHE_{v,b}(G_\phi(\cX_n,r_n))}{n^2 r_n^{d+1}} \right),
\label{0107b}
\eea
and
\bea
\limsup_{n \to \infty} 
\left( \frac{ TV_{\phi, \tr_n} ({\bf 1}_{B_n})}{
\Bal_{n,v,b}(\cY'_{n,B}) } \right)
\leq  \sigma_{\phi} \CHE_{v,b}(D,\rho) < \infty.
\label{0107c}
\eea
\end{lemm}
\begin{proof}
Set $V_n := \Vol_{n,v}(\X_n)$ as before.
For all but finitely many $n \in \N$ we have by
Lemma \ref{lemCher} that  $n \geq N$, 
and by Lemma \ref{Npplem}
 that $n \in \cNstar$ (so that $ \Vol_{n,v}(\cY_n) /V_n \in [\gamma_n,1/2]$),  
and by Lemma  \ref{Apartheid} that  $g_n(\cY'_n) =0$ 
  so that for each $i \in S_n$ the
 box $Q_{i,n} $ is either black or white but not both
(with respect to $\cY'_n$), and
 the sets $B_n$ and $W_n:= D \setminus B_n$ are non-empty.
Consider such $n$.

We claim next that whichever value of $ (v,b) \in \{1,2\}^2$ we
are considering, we have
for large enough $n $ that
 \bea
\Bal_{n,v,b}(\cY'_n)/\Bal_{n,v,b}(\cY_n) \geq 1- \gamma_n.
\label{1218a}
\eea
If $v=1$ this follows from 
 (\ref{1013e}) and the fact that $n \in \cN^*$.
If $v=2$, to see (\ref{1218a}) we use also the fact
that by (\ref{1013e}) and (\ref{1216a}) we have
$$
|\Vol_{n,2,b}(\Y'_n) - \Vol_{n,2,b}(\Y_n)|\ll n^2 r_n^d \gamma_n^2
$$
while $\Vol_{n,2}(\cY_n) \geq c \gamma_n n^2 r_n^d$ for some constant
$c >0$, by
the fact that $n \in \cNstar$, and
 (\ref{RGGedgeLLN}).

Hence for large enough $n$, 
by (\ref{1121f}) 
and (\ref{1127a})
we have
\bea
\CHE_{v,b}(G_\phi(\cX_n,r_n))
\geq
\frac{ (1- \gamma_n) 
\Cut'_{n,\phi}(\cY'_n)}{\Bal_{n,v,b}(\Y'_n)}.
\label{1015b}
\eea
 Let $\Y'_{n,B}:= \cY'_n \cap B_n$ and
$\Y'_{n,W} := \Y'_n \cap W_n$. 
 Let $\partial_{n,\phi}^W(\cY'_n)$ be the
total $\phi_n$-weight of within-box edges
 from $\Y'_{n,W}$ to $\X_n \setminus \Y'_n$.
Then
\bean
\frac{ \Cut'_{n,\phi} (\cY'_n)}{\Vol_{n,v}(\cY'_n)} \geq
 \frac{
Z_n + \partial_{n,\phi}^W (\cY'_n) }{\Vol_{n,v}(\cY'_{n,B})+ \Vol_{n,v}(\cY'_{n,W}) }
\geq  \min \left(
 \frac{
Z_n }{ \Vol_{n,v}(\cY'_{n,B}) },
\frac{
 \partial_{n,\phi}^W (\cY'_n)}{ \Vol_{n,v}(\cY'_{n,W}) }
\right).
\eean
Hence
\bea
\frac{
( n^2 r_n^{d+1})^{-1}
 \Cut'_{n,\phi} (\cY'_n)}{\Bal_{n,v,1}(\cY'_n)} \geq
\min \left(
 \frac{
( n^2r_n^{d+1})^{-1} Z_n }{
 \Vol_{n,v}(\cY'_{n,B})/V_n  },
\frac{
(n^2r_n^{d+1})^{-1}
 \partial_{n,\phi}^W (\cY'_n)}{ \Vol_{n,v}(\cY'_{n,W})/V_n }
\right).
\label{1121h}
\eea
Each vertex $x \in \cY'_{n,W}$ has at least $\gamma_n^d(1-2\gamma_n) \rhomin n r_n^d$
within-box neighbours in $\cX_n \setminus \cY'_n$.
 Therefore there is a constant $c'> 0 $ such that
for $n$ large enough,
 the second ratio inside the minimum in the
right hand side of (\ref{1121h})
is at least $ c' \gamma_n^{d} (1-2 \gamma_n) \rhomin  r_n^{-1} \phi(0) $;
 we may take $c' =1/2$ for $v=1$  
 and otherwise use the fact that by (\ref{1216a}) and (\ref{RGGedgeLLN}) we have
$$
\Vol_{n,2}(\cY'_{n,W})/\Vol_{n,2}(\cX_n) \leq
{\rm const.} \times \Vol_{n,1}(\cY'_{n,W})/\Vol_{n,1}(\cX_n). 
$$
Hence this second ratio tends to infinity by (\ref{gammacond}).
By (\ref{1015b}) and (\ref{1013b}),  
if $b=1$ then the left hand side of (\ref{1121h}) is bounded, so
for large enough $n $ the minimum is achieved by the first of the two ratios in
the right hand side of (\ref{1121h}).

Also,
\bea
\frac{ \Cut'_{n,\phi} (\cY'_n)}{\left(\frac{\Vol_{n,v}(\cY'_n)}{V_n}\right)
 \left(1- 
\frac{\Vol_{n,v}(\cY'_n)}{V_n} \right) }\geq
\frac{Z_n + \partial_n^W(\cY'_n)}{ \left(
 \frac{\Vol_{n,v}(\Y'_{n,B})}{V_n} + 
\frac{\Vol_{n,v}(\Y'_{n,W})}{V_n}
\right)\left(1- \frac{\Vol_{n,v} (\Y'_n) }{V_n} \right) } 
\nonumber \\
\geq \min \left( 
\frac{Z_n}{
\left( \frac{\Vol_{n,v}(\Y'_{n,B})}{V_n}  
\right)\left(1- \frac{\Vol_{n,v} (\Y'_n) }{V_n} \right) } 
,
\frac{\partial_n^W(\cY'_n)}{\frac{\Vol_{n,v} (\Y'_{n,W})}{V_n}}
 \right),
\label{1218c}
\eea
and similarly to before, if $b=2$ then for large
enough $n$ the minimum is achieved by the first
 term. Thus using (\ref{1121h}) for $b=1$ and (\ref{1218c})
for $b=2$, in both cases
we have for large enough $n$  that
\bean
\frac{
 \Cut'_{n,\phi} (\cY'_n)}{\Bal_{n,v,b}(\cY'_n)} \geq
\frac{ 
Z_n}{
\Bal_{n,v,b}(\cY'_{n,B})
}.
\eean
Therefore using (\ref{1015c}) followed by
(\ref{1015b}) we have that
\bean
 \frac{ TV_{\phi, \tr_n} ({\bf 1}_{B_n})}{
\Bal_{n,v,b}(\cY'_{n,B}) } 
 \leq
 \frac{ TV_{\phi, \tr_n} ({\bf 1}_{B_n})
\Cut'_{n,\phi}(\cY'_n)
}{
Z_n  
\Bal_{n,v,b} (\cY'_n)} 
\leq
\frac{ (2+o(1)) (n^2 r_n^{d+1} )^{-1} \Cut'_{n,\phi}(\cY'_n)}{\Bal_{n,v,b}
(\cY'_n)} 
\\
\leq (2 +o(1)) 
(n^2 r_n^{d+1})^{-1} 
\CHE_{v,b}(G_\phi(\cX_n,r_n)) 
.
\eean
This gives us (\ref{0107b}), and then (\ref{0107c})
follows from
 Proposition \ref{propupper}.
\end{proof} 
Using Lemma \ref{lemEgoroff} we obtain the following.
\begin{lemm}
\label{lemBnice}
Let $B_n$ and $\cY'_{n,B}$ be as in the preceding lemma.
Then:

(i) It is the case that $\limsup_{n \to \infty} \nu(B_n) < 1$.

(ii)
For every subsequence of $\N$ such that $\nu(B_n)$
is bounded away from zero along the subsequence,
we have along that subsequence that
\bea
\lim_{n \to \infty} \left( \frac{ 
\Bal_{n,v,b}(\cY'_{n,B})}{\Bal_{\nu,v,b}(B_n) } 
\right) 
=1 .
\label{0106d}
\eea
\end{lemm}
\begin{proof}
(i)
It is enough to prove that for every infinite subsequence
$\cN$ of $\N$ with
$\inf_{n \in \cN} \nu(B_n) >0$ we have
   $\limsup \nu(B_n) <1 $ as $n \to \infty$
through $\cN$.
By (\ref{1013e}) and (in the case $v=2$) (\ref{1216a}) we have
$$
\frac{\Vol_{n,v}(\cY'_n) - \Vol_{n,v}(\cY_n) }{ \Vol_{n,v}(\cX_n) } \to 0,
$$
and since we assume $\Vol_{n,v}(\cY_n) \leq \Vol_{n,v}(\cX_n) /2$,
by Lemma \ref{lemEgoroff}
we have
for large enough $n$  that
\bean
(2/3) \geq \frac{\Vol_{n,v}(\cY'_n)}{\Vol_{n,v}(\cX_n)}
 \geq \frac{\Vol_{n,v}(\cY'_{n,B})}{\Vol_{n,v}(\cX_n)}
= (1+o(1))  \left( \frac{ \int_{B_n} \rho(x)^v dx }{
\int_D \rho(x)^v dx} \right), 
\eean
which gives us part (i).

(ii)
Let $\cN \subset \N$ be  infinite with
  $\inf_{n \in \cN}  \nu(B_n) >0$.
  By Lemma \ref{lemEgoroff} we have
\bea
\lim_{n \to \infty,n \in \cN} \left( 
\frac{\Vol_{n,v}(\cY'_{n,B}) \Vol_{\nu,v}(D) }{
\Vol_{n,v}(B_n) \Vol_{n,v}(\X_n)}
\right) =1.
\label{0108e}
\eea
Set $W_n := D \setminus B_n$ and $\cY'_{n,W}:= \cY'_n \cap W_n$.
By part (i), $\nu(W_n)$ is bounded away from zero and therefore
by  applying  Lemma \ref{lemEgoroff} again,
 we have
\bea
\lim_{n \to \infty,n \in \cN} \left( 
\frac{\Vol_{n,v}(\cY'_{n,W}) \Vol_{\nu,v}(D) }{
\Vol_{n,v}(W_n) \Vol_{n,v}(\X_n)}
\right) =1.
\label{0109a}
\eea
Moreover, since the number of white points in black boxes is
at most $n\gamma_n $, using also  (\ref{1216a})
(in the case $v=2$) we have
$$
\frac{
\Vol_{n,v}(\X_n \setminus \cY'_{n,B}) - \Vol_{n,v}(\cY'_{n,W}) 
  }{
\Vol_{n,v}(\cX_n)
}
= \frac{ \Vol_{n,v}( (\cX_n \setminus \cY'_n) \cap B_n)}{\Vol_{n,v}(\cX_n) }
\leq {\rm const.} \times \gamma_n 
$$ 
which tends to zero, 
and hence by (\ref{0109a}) we have
\bean
\lim_{n \to \infty,n \in \cN} \left( 
\frac{\Vol_{n,v}(\cX_n \setminus \cY'_{n,B}) \Vol_{\nu,v}(D) }{
\Vol_{n,v}(W_n) \Vol_{n,v}(\X_n)}
\right) =1.
\eean
By using this, along with (\ref{0108e}), we can obtain (\ref{0106d}).
\end{proof}

\begin{lemm}
\label{lemsubseq}
Let $B_n$ and $\cY'_{n,B}$ be as in the preceding lemma.
For any subsequence of $\N$ there 
 a further subsequence along which the
functions 
$u_n$
converge in $L^1(\nu)$ to ${\bf 1}_A/\Bal_{\nu,v,b}(A)$ 
for some $A \in \cB(D)$ with $0 < \nu(A) <1$.
\end{lemm}
\begin{proof}
We claim that it suffices to prove that
\bea
\limsup_{n \to \infty} \left( \frac{ 
\Bal_{n,v,b}(\cY'_{n,B})}{\Bal_{\nu,v,b}(B_n) } 
\right) < \infty.
\label{0106a}
\eea
Indeed,  suppose 
 (\ref{0106a}) holds and for 
$n \in \N$
 define the
function $u_n :=
{\bf 1}_{B_n}/\Bal_{\nu,v,b}({B_n})$.  Then by (\ref{0106a}) and (\ref{0107c})
the sequence $TV_{\phi,\tr_n}(u_n)$ is bounded, so we can 
apply Lemma \ref{LemGam}  (iii), to deduce that for any subsequence of $\N$
 there exists a further subsequence along which
the functions ${\bf 1}_{B_n}/\Bal_{\nu,v,b}(B_n)$ converge
 in $L^1(\nu)$ to a limiting function
which must necessarily be of the form  ${\bf 1}_A/\Bal_{\nu,v,b}(A)$ with
$A \in \cB(D)$ and $0< \nu(A)<1$ (see Lemma 6 of \cite{consistency}).

It remains to prove (\ref{0106a}).
First suppose $v=1$. By (\ref{Cher1}) we have
for large enough $n$ that
\bea
|\cY'_{n,B}| \leq 2 n \nu(B_n).
\label{0106b}
\eea
Also, by (\ref{Cher1}) and Lemma \ref{lemBnice},
for large enough $n$ we have
$$
|\cX_n \setminus \cY'_{n,B}| \leq 2n \nu(W_n) + \gamma_n n \nu(B_n) 
\leq 3n \nu(W_n),
$$
and combined with (\ref{0106b}) this gives us
(\ref{0106a}) in the case $v=1$ (either for $b=1$ or $b=2$).

Now suppose $v=2$. Set $V_n := \Vol_{n,2}(\cX_n)$. Using
(\ref{0106b}) and (\ref{1216a}) we have that   
\bea
V_n^{-1} \Vol_{n,2} (\cY'_{n,B}) \leq {\rm const.} \times \nu(B_n).
\label{0106c}
\eea
Since $\nu(B_n)$ is bounded away from 1 by Lemma \ref{lemBnice} (i),
using (\ref{1216a}) and (\ref{Cher1}) 
we have
$$
\limsup \left( \frac{ \Vol_{n,2}(\cX_n \setminus \cY'_{n} ) }{
\nu(W_n) V_n } \right) < \infty
$$
and combined with (\ref{0106c}) this gives us (\ref{0106a}) for
$v=2$.
\end{proof}

\begin{proof}[Proof of Theorem \ref{mainthm}]
By Lemma \ref{lemsubseq}, 
for any subsequence of 
$\N$
 there exists a further subsequence along which
the functions ${\bf 1}_{B_n}/\Bal_{\nu,v,b}(B_n)$ converge
 in $L^1(\nu)$ to a limiting function
 of the form  ${\bf 1}_A/\Bal_{\nu,v,b}(A)$ with
$A \in \cB(D)$ and $0< \nu(A)<1$. 
Then $\nu(B_n)$ is bounded away from zero.

 Then by part (i) of Lemma \ref{LemGam},
and Lemma \ref{lemBnice}, we have along this subsequence that
\bean
 \sigma_\phi TV ( {\bf 1}_{A}/ \Bal_{\nu,v,b}(A))
\leq
  \liminf_{n \to \infty}  TV_{\phi,\tr_n}({\bf 1}_{B_n}/\Bal_{\nu,v,b}
(B_n))
 \nonumber \\
\leq  \liminf_{n \to \infty}
TV_{\phi,\tr_n}({\bf 1}_{B_n} / \Bal_{n,v,b}(\cY'_{n,B}) ).
\eean
Therefore
by (\ref{0107b}),
followed by (\ref{1013b}),
 we have along this subsequence that
\bea
 (\sigma_\phi/2)TV ( {\bf 1}_{A}/ \Bal_{\nu,v,b}(A))
\leq \liminf_{n \to \infty} \left(
  \frac{\CHE_{v,b}(G_\phi(\X_n,r_n))}{n^2 r_n^{d+1}} \right)
\nonumber
 \\
\leq \limsup_{n \to \infty} \left(
  \frac{\CHE_{v,b}(G_\phi(\X_n,r_n))}{n^2 r_n^{d+1}} \right)
\leq  \CHE_{v,b}(D,\rho) \sigma_\phi/2.
\label{1017b}
\eea
By the definition (\ref{Chee1}) the 
 inequalities in (\ref{1017b}) are all  equalities and the set $A$ is a
minimiser in (\ref{Chee1}).
This gives us the asserted convergence (\ref{1011a}).
\end{proof}

\begin{proof}
[Proof of Theorem \ref{Weaktheo}]
To prove this we re-examine the preceding proof.
For each $n \in \N$,
let $\cY_n \subset \cX_n$ be a minimiser as
in the definition (\ref{Chee0}) of $\CHE_{v,b}(G_\phi(\cX_n,r_n))$
with $\Vol_{n,v}(\cY_n) / \Vol_{n,v}(\cX_n) \leq 1/2 ,$ 
as before.

Let $\cY'_n$ and $B_n$ be as in the previous proof.
As shown there, for every subsequence there is a further subsequence
along which ${\bf 1}_{B_n} \to {\bf 1}_A$ in $L^1(\nu)$ for some optimising
set $A$.

Let $\mu_n := \sum_{y \in \cY_n} n^{-1} \delta_{y} $ and
let $\mu'_n := \sum_{y \in \cY'_n} n^{-1} \delta_{y} $.
To demonstrate (\ref{1018d}),
 we need
to show that the sequence $(\mu_n)$ of measures converges weakly to the restriction of $\nu$ to  $A$.
By the Portmanteau theorem \cite{Bill}, it is enough to show that
for any uniformly continuous function $f$ on $D$ we have
$\mu_n(f) \to \nu(f {\bf 1}_A)$. Since any such $f$ is bounded, we have by
(\ref{1013e}) 
 that $n^{-1} |\mu'_n(f) - \mu_n(f)| \to 0$.

On $W_n := D \setminus B_n$ the density of points relative
to the measure $n \nu$ is at most $\gamma_n$; that is,
$\mu'_n(W_n) \leq \gamma_n \nu(D )$
which tends to zero.

Since $f$ is uniformly continuous, given $\eps >0$ we can find $n_0$ such that
for $n \geq n_0$ we have for all $i \in S_n$ that $\overline{f}_{i,n} := \sup_{Q_{i,n}} f $
and $\underline{f}_{i,n} := \inf_{Q_{i,n}} f$ satisfy $\overline{f}_{i,n} - \underline{f}_{i,n} < \eps$.
Then for $n \geq n_0$, setting $\fmax = \sup_{x \in D}f(x)$, we have
\bean
\mu'_n(f)
\leq \left( \sum_{\{i: Q_{i,n} \subset B_n\}} (1+\gamma_n)  \nu(Q_{i,n} ) \overline{f}_{i,n} \right)
+  \fmax \mu'_n(W_n)
\\
\leq (1+\gamma_n) \int_{A}(f(x)  + \eps) \nu(dx) + o(1),
\eean
and therefore
\bea
\limsup_{n \to \infty} \mu'_n(f)
\leq \int_{B}(f(x)  + \eps) \nu(dx).
\label{1018c}
\eea
Also,
\bean
\mu'_n(f)
\geq  \sum_{\{i: Q_{i,n} \subset B_n\}} (1-2\gamma_n)  \nu(Q_{i,n} ) \underline{f}_{i,n}
\geq (1-2\gamma_n) \int_{B_n}(f(x)  - \eps) \nu(dx)
\eean
so that
\bea
\liminf_{n \to \infty} \mu'_n(f)
\geq \int_{A}(f(x)  - \eps) \nu(dx).
\label{0502a}
\eea
Combining this with (\ref{1018c}) gives us  (\ref{1018d}).
\end{proof}

\begin{proof}[Proof of Corollary \ref{corowk}]
Assume that the hypotheses of Theorem \ref{mainthm} apply,
and also that the minimising set
$A$   in the definition (\ref{Chee1}) of $\CHE_{v,b}(D,\rho)$
is unique, up to complementation and adding or removing sets of
$(d-1)$-dimensional measure zero.
We shall use the following. Set $\nu_n: = \sum_{y \in \X_n } n^{-1} \delta_y$.
Given any uniformly continuous function $f$ on $D$, 
similarly to (\ref{1018c}) and (\ref{0502a}) 
it can be shown that
\bea
\nu_n(f) \to \nu(f).
\label{0502b}
\eea
We also use the fact that the topology of 
 weak convergence of probability measures on 
$D$ is metrizable by the Prohorov metric (here denoted $d_\pi$) 
on the space of such measures. 
See \cite[page 72]{Bill}, where a definition
of this metric can also be found.

Let ${\cal Y}_n$ be a sequence of minimisers 
in the definition (\ref{Chee0}) of $\CHE_{v,b}(G_\phi(\cX_n,r_n))$.
 Then we claim that
\bea
\min( d_\pi( \mu_n, \nu|_A ), d_\pi(\mu_n,\nu|_{A^c} )) \to 0.
\label{0426d}
\eea
It is straightforward to deduce this from
  the part of Theorem \ref{Weaktheo} already proved,
 along with our uniqueness assumption regarding $A$, 
noting also that if we add or remove a Lebesgue-null  set
to/from $A$, the measure $\nu|_A$ is unchanged.
Now
take
$$
j(n) = \begin{cases} 1 \mbox{ if } d(\mu_n,\nu|_A) \leq d(\mu_n,\nu|_{A^c}) 
\\ 
0 \mbox{ otherwise}.
\end{cases}
$$

On the sequence of $n$ for which $j(n) =1$ (if this sequence is
infinite), by (\ref{0426d}) we have $d_\pi(\mu_n, \nu|_{A}) \to 0$ so 
$\mu_n $ converges weakly to $ \nu|_{A}$.

On the sequence of $n$ for which $j(n) =0$ (if this sequence is
infinite), by (\ref{0426d}) we have $d_\pi(\mu_n, \nu|_{A^c}) \to 0$ so
 $\mu_n $ converges weakly to $\nu|_{A^c}$.
Then using (\ref{0502b}),
we have for any uniformly continuous function $f$ on $D$
that $(\nu_n - \mu_n)(f) \to \nu|_{A}(f)$, so $\nu_n -\mu_n$ converges
weakly to $\nu|_A$.

Putting the last two paragraphs together gives us the desired
conclusion. 
\end{proof}

\section{The bisection problem}
\label{secbisec}
\allco

In this section we prove  Theorem \ref{thmBISBHH}.  The result is
immediate from Lemmas \ref{lembisub} and \ref{lembislb}.
We assume throughout this section that the assumptions
of Theorem \ref{thmBISBHH} apply.
\begin{lemm}
\label{lembisub}
It is the case that
\bea
\limsup_{n \to \infty} \left(
 \frac{ \MBIS(G_\phi(\cX_n,r_n)) }{n^2 r_n^{d+1} } \right)
 \leq (\sigma_{\phi}/2)
\MBIS_\nu(D),
~~~ a.s.
\label{eqbisub}
\eea
\end{lemm}
\begin{proof}
Let $A \in \cB(D) $ with $\nu(A) =1/2$.
Set $\cY_n = \cX_n \cap A$.
Then by (\ref{1103a}),
\bea
(n^2 r_n^{d+1})^{-1} \Cut_{n,\phi}(\cY_n) \to (\sigma_\phi/2) TV({\bf 1}_A). 
\label{1018f}
\eea
Also $n^{-1} |\cY_n| \to 1/2$ by the strong law of large numbers,
but of course this does not tell us that $|\cY_n| = \lfloor n/2 \rfloor$.
Set $M_n:= |\cY_n| - \lfloor n/2 \rfloor$.
Using the Chernoff bounds (\ref{CherUB}) and (\ref{CherLB}),
Taylor's Theorem (as in the proof of Lemma \ref{lemCher})
and the Borel-Cantelli lemma, we have
almost surely that for large enough $n$, 
\bea
|M_n | \leq 3 (n \log n)^{1/2}
\label{1218e}
\eea

As in the preceding section, we shall say that points in
$\cY_n$ are black and points in $\cX_n \setminus \cY_n$ are
white.
%
%
%
If $M_n <0$ let us
pick $|M_n|$  points in $\cX_n \setminus \cY_n$, 
and add
them to $\cY_n$
(i.e., change their colour from white to  black).
If $M_n >0$,
pick $M_n$   points in $\cY_n$, and remove them from $\cY_n$ (i.e., change
their colour from black to white).
 In
both cases let $\cY'_n$ be the resulting set of black points.
Then $|\cY'_n| = \lfloor n/2 \rfloor$.

By  (\ref{1216a}),
the total weight  of cut edges created or destroyed  by changing
from $\cY$ to $\cY'$  is at most
a constant times $nr_n^d |M_n|$.
Therefore using (\ref{1218e}) we have 
\bean
 (n^2 r_n^{d+1} )^{-1} | \Cut_{n,\phi}(\cY'_n) - \Cut_{n,\phi}(\cY_n)|
= O(  (n^{-1} \log n)^{1/2} r_n^{-1} ) 
\\ 
= O(( n r_n^d /\log n)^{-1/2} r_n^{(d-2)/2}) 
\eean
which tends to zero by the assumptions $nr_n^d \gg \log n$ and
$d \geq 2$.
Combined with (\ref{1018f}) this shows that
\bean
\limsup_{n \to \infty} \left( \frac{\MBIS(G_\phi(\X_n,r_n))}{n^2 r_n^{d+1} } \right)
\leq
\limsup_{n \to \infty} \left(
 \frac{\Cut_{n,\phi}(\cY'_n)}{n^2 r_n^{d+1} }  \right) =
 (\sigma_\phi/2) TV({\bf 1}_A).
\eean
Taking the infimum over all $A$ and using
 (\ref{MBISDdef}),  this gives us (\ref{eqbisub}).
\end{proof}

\begin{lemm}
\label{lembislb}
It is the case that
\bea
\liminf_{n \to \infty}
\left(
 \frac{ \MBIS(G_\phi(\cX_n,r_n)) }{n^2 r_n^{d+1} }
\right)
 \geq (\sigma_{\phi}/2)
\MBIS_\nu(D),
~~~ a.s.
\label{1218f}
\eea
\end{lemm}
\begin{proof}
We argue  similarly to the proof in Section \ref{seclo}.
For each $n \in \N$ let  $\cY_n$ be a bisection
 of $\X_n$ (i.e. a subset with
 $\lfloor n/2 \rfloor$ elements) that achieves
the minimum in the definition (\ref{0107a}).
Define the boxes $Q_{i,n}$ as in Section \ref{seclo},
 and define black, grey and white
boxes as we did there.
By Lemma \ref{Apartheid}, there is a  set $\cY'_n$ satisfying
(\ref{1013e}) and inducing no grey boxes, such that
 $$
\MBIS(G_\phi(\cX_n,r_n)) =
\Cut_{n,\phi}(\cY_n) \geq 
\Cut'_{n,\phi}(\cY'_n) \geq Z_n, 
$$
where $Z_n$ denotes the contribution to $ \Cut'_{n,\phi}(\cY'_n)$
from edges with exactly one endpoint in $B_n$, and 
 $B_n$ denotes the union of the black boxes induced by $\cY'_n$.
Then by
 (\ref{1015c}),
\bea
%
\limsup_{n \to \infty} \left( \frac{
\MBIS(G_\phi(\cX_n,r_n))}{ n^2 r_n^{d+1}} \right) \geq
(1/2)
\limsup_{n \to \infty} TV_{\phi,\tr_n}({\bf 1}_{B_n}),
\label{0107d}
\eea
and the left side of (\ref{0107d}) is finite by Lemma \ref{lembisub}.
By the compactness property (Lemma \ref{LemGam} (iii)), 
for any infinite subsequence $\cN \subset \N$, we may 
find an infinite  subsequence $\cN' \subset \cN$ such
that as $n \to \infty$ through $\cN'$, 
the fucntions
${\bf 1}_{B_n} $ converge in $L^1(\nu)$ to a limit, necessarily of the
form ${\bf 1}_B$ for some $B \in \cB(D)$.

Using (\ref{1013e}) and
the fact that the original $\cY_n$ was a bisection, we have that
$n^{-1} |\Y'_n| \to 1/2$. Then using
(\ref{Cher1}), (\ref{Cher2}) and the fact
 that we take $ \gamma_n \to 0$ so
the proportion of black vertices in white boxes or white vertices in black
boxes vanishes, we have that
$\nu(B_n) \to 1/2$. 
   Therefore $\nu(B)=1/2$. 

By the definition in (\ref{MBISDdef}),  and
 the liminf lower bound from Lemma \ref{LemGam} (i), we have
as $n \to \infty$ through $\cN'$ that
$$
(\sigma_\phi/2) \MBIS_{\nu}(D) 
\leq
(\sigma_\phi/2) TV ({\bf 1}_B)  \leq \liminf_{n \to \infty} (1/2) 
TV_{\phi,\tr_n} ({\bf 1}_{B_n}), 
$$
and since (\ref{0107d}) still holds with $\limsup$ replaced
by $\liminf$ on both sides, we thus obtain (\ref{1218f}).
%
\end{proof}

{\bf Acknowledgement.}
We thank 
 Nicol\'as Garc\'ia Trillos and Ery Arias Castro for
answering some questions regarding their work on
this subject.


The research leading to this paper  was  partially carried out
during an extended visit of the second author to Utrecht University.
He thanks the Department of Mathematics at Utrecht University for
its hospitality.

{}

\end{document}